\documentclass[11pt]{article}

\title{Critical exponent for nonlinear wave equations with frictional and viscoelastic damping terms}
\author{
Ryo Ikehata\thanks{ikehatar@hiroshima-u.ac.jp},\\
Department of Mathematics, \\
Graduate School of Education, Hiroshima University\\
Higashi-Hiroshima 739-8524, Japan \\
\ \\
Hiroshi Takeda\thanks{Corresponding author: h-takeda@fit.ac.jp},\\
Department of Intelligent Mechanical Engineering, \\
Faculty of Engineering, Fukuoka Institute of Technology, \\
3-30-1 Wajiro-higashi, Higashi-ku, Fukuoka, 811-0295 JAPAN 
}
\date{}

\usepackage{amssymb,amsmath,amsthm}
\usepackage[dvips]{graphicx}
\usepackage{bm} 
\newcommand{\R}{\mathbb R}

\newcommand{\K}{\mathcal{K}}

\newcommand{\supp}{\mathop{\mathrm{supp}}\nolimits}

\newtheorem{thm}{Theorem}[section]

\newtheorem{prop}[thm]{Proposition}
\newtheorem{lem}[thm]{Lemma}
\theoremstyle{remark}
\newtheorem{rem}[thm]{Remark}

\theoremstyle{definition}

\begin{document}
\maketitle

\numberwithin{equation}{section}
%%%%%%%%%%%%
\begin{abstract}
In this paper,  we study the Cauchy problem for a nonlinear wave equation with 
frictional and viscoelastic damping terms.
Our aim is to obtain the threshold, to classify the global existence of solution for small data or 
the finite time blow-up pf the solution, 
with respect to the growth order of the nonlinearity.  
\end{abstract}

\newpage
%%%%%%%%%%%%%%%%%%%%%%%%%%%%%%%%%%%%%%%%%%%%%%%%%%%%%%%%%%%%%%%%%%%%%%%%%%%%%%%%%%%%%%%%%%%%%%%%%%%%%%%%%%%%%%%%%%%%%%%%%%%%%%%%%%%%%%%%%%%%%%%
\section{Introduction}

In this paper we consider the Cauchy problem for a wave equation with two types of damping terms 
\begin{equation} \label{eq:1.1}
\left\{
\begin{split}
& \partial_{t}^{2} u -\Delta u + \partial_{t} u -\Delta \partial_{t} u =f(u), \quad t>0, \quad x \in \R^{n}, \\
& u(0,x)=u_{0}(x), \quad \partial_{t} u(0,x)=u_{1}(x) , \quad x \in \R^{n}, 
\end{split}
\right.
\end{equation}
where $u_{0}(x)$ and $u_{1}(x)$ are given functions.\\

After A. Matsumura \cite{Ma} has established a pioneering basic decay estimates to the linear equation
\begin{equation} \label{eq:1.1.0}
\left\{
\begin{split}
& \partial_{t}^{2} u -\Delta u + \partial_{t} u = 0, \quad t>0, \quad x \in \R^{n}, \\
& u(0,x)=u_{0}(x), \quad \partial_{t} u(0,x)=u_{1}(x) , \quad x \in \R^{n}, 
\end{split}
\right.
\end{equation}
many mathematicians have concentrated on solving a typical important nonlinear problem of the semi-linear damped wave equation 
\begin{equation} \label{eq:1.1.1}
\left\{
\begin{split}
& \partial_{t}^{2} u -\Delta u + \partial_{t} u = \vert u\vert^{p}, \quad t>0, \quad x \in \R^{n}, \\
& u(0,x)=u_{0}(x), \quad \partial_{t} u(0,x)=u_{1}(x) , \quad x \in \R^{n}, 
\end{split}
\right.
\end{equation}
and at that case we necessarily remind of the Fujita critical exponent. That is, there exists a real number $p_{F} \in (1,\infty)$ such that if $p > p_{F}$, then the corresponding Cauchy problem (1.3) has a small global in time solution $u(t,x)$ for small initial data $[u_{0},u_{1}]$, while in the case when $p \in (1,p_{F}]$, the corresponding problem does not have any nontrivial global solutions. The number $p_{F}$ is called as the Fujita critical exponent, and nowadays it is well-known that $p_{F} = 1+ \displaystyle{\frac{2}{n}}$. Even if we restricted to the Cauchy problem case in ${\R^{n}}$, one can cite so many related research papers due to \cite{HKN}, \cite{H}, \cite{HO}, \cite{Ik-3}, \cite{IMN}, \cite{IT}, \cite{K}, \cite{KU}, \cite{MK}, \cite{MN}, \cite{Na}, \cite{N-2}, \cite{T}, \cite{TY}, \cite{Z} and the references therein. It should be emphasized that the first success to find out the Fujita number in a complete style for all $n \geq 1$ is in the work due to Todorova-Yordanov \cite{TY}. Anyway, these results are based on an important recognition that the asymptotic profile as $t \to +\infty$ of the solution of the linear equation (1.2) is a constant multiple of the Gauss kernel or a solution of the corresponding heat equation with an appropriate initial data. This type of diffusion phenomenon is discussed in \cite{CH}, \cite{DE}, \cite{DR}, \cite{HM}, \cite{HO}, \cite{Ik-1}, \cite{IN}, \cite{K}, \cite{KU}, \cite{MN}, \cite{Na}, \cite{N}, \cite{N-2}, \cite{RTY}, \cite{said}, \cite{SW}, \cite{T-2}, \cite{W} and the references therein. \\

On the other hand, if the problem \eqref{eq:1.1.1} is replaced by the following strongly damped wave equation case,
\begin{equation} \label{eq:1.1.2}
\left\{
\begin{split}
& \partial_{t}^{2} u -\Delta u -\Delta\partial_{t} u = \mu f(u), \quad t>0, \quad x \in \R^{n}, \\
& u(0,x)=u_{0}(x), \quad \partial_{t} u(0,x)=u_{1}(x) , \quad x \in \R^{n}, 
\end{split}
\right.
\end{equation}
there seems not so many related research papers at present. In the case when the linear equations are concerned with $\mu = 0$ in \eqref{eq:1.1.2}, one has two pioneering papers due to Ponce \cite{Ponce} and Shibata \cite{Shibata}, in which they studied $L^{p}-L^{q}$ decay estimates of the solutions to \eqref{eq:1.1.2} with $\mu = 0$. Quite recently, Ikehata-Todorova-Yordanov \cite{ITY} and Ikehata \cite{Ik-4, IR} have caught an asymptotic profile of solutions to problem \eqref{eq:1.1.2} with $\mu = 0$, and in fact, its profile is so called the diffusion waves, which are well-studied in the field of the Navier-Stokes equation case. In this case, an oscillation property occurs in the low frequency region, while in the usual frictional damping case \eqref{eq:1.1.0} one cannot observe any such oscillation properties. There is a big difference between \eqref{eq:1.1.0} and \eqref{eq:1.1.2} with $\mu = 0$. In connection with this,  several decay estimates for wave equations with structural damping, which interpolate \eqref{eq:1.1.1} and \eqref{eq:1.1.2} are extensively studied to the equation 
\begin{equation} \label{eq:1.1.3}
\partial_{t}^{2} u -\Delta u + b(t)(-\Delta)^{\theta}\partial_{t} u = \mu f(u),
\end{equation}
where $\theta \in [0,1]$ and $\mu \geq 0$ in the papers due to \cite{CLI}, \cite{DE}, \cite{DEP}, \cite{DR}, \cite{INatsume}, \cite{LIC}, \cite{LR}, \cite[$\mu = 0$, $b(t) \equiv 1$, $\theta = 0$, exterior domain case]{racke}, \cite[$\mu= 0$, $\theta = 0$]{JW}, and the references therein. At this stage, if one considers the original problem \eqref{eq:1.1}, which has two types of damping terms, a natural question arises that which is dominant as $t \to +\infty$, frictional damping or viscoelastic one?\, About this fundamental question, quite recently Ikehata-Sawada \cite{IS} have studied the problem  
\begin{equation} \label{eq:1.1.4}
\left\{
\begin{split}
& \partial_{t}^{2} u -\Delta u + \partial_{t} u -\Delta \partial_{t} u = 0, \quad t>0, \quad x \in \R^{n}, \\
& u(0,x)=u_{0}(x), \quad \partial_{t} u(0,x)=u_{1}(x) , \quad x \in \R^{n}, 
\end{split}
\right.
\end{equation}
and have shown that $u(t,x) \sim (P_{0}+P_{1})G_{t}(x)$ as $t \to +\infty$ in the $L_{x}^{2}$-sense, where $P_{j} := \displaystyle{\int_{{\R^{n}}}}u_{j}(x)dx$ ($j = 0,1$), and $G_{t}(x) = \displaystyle{\frac{1}{(\sqrt{4\pi t})^{n}}}e^{-\frac{\vert x\vert^{2}}{4t}}$ is the usual $n$-dimensional Gauss kernel. From this observation on the linear equation \eqref{eq:1.1.4}, one can know that as $t \to +\infty$, the dominant term is frictional damping. So, based on this result to the linear problem \eqref{eq:1.1.4}, a natural question arises again such as: if one considers the semi-linear problem \eqref{eq:1.1}, does the critical exponent coincide with the Fujita type $p_{F}$?\, Our main purpose of this paper is to answer this conjecture.\\

Our first result below is concerned with the existence of the global solution of \eqref{eq:1.1} satisfying the suitable decay properties.
\begin{thm} \label{Thm:1.1}
Let $n=1,2$, $p> 1 + \displaystyle{\frac{2}{n}}$.
Assume that 
$(u_{0}, u_{1}) \in L^{1} \cap H^{k_{0}}\times L^{1} \cap L^{2}$ 
with sufficiently small norm, where 
\begin{align}
k_{0} = 
\begin{cases}
& 1, \quad n=1, \\
& 1+ \varepsilon, \quad n=2
\end{cases}
\end{align}
for an arbitrary $\varepsilon \in [0,2)$.
Then, there exists a unique global solution of \eqref{eq:1.1} in the class $C([0, \infty);H^{k_{0}})$ 
satisfying
\begin{equation}
\begin{split}
\| |\nabla_{x}|^{k} u(t) \|_{L^{2}(\R^{n})} \le C(1+t)^{-\frac{n}{4}-\frac{k}{2}}.
\end{split}
\end{equation}
for $k \in [0, k_{0}]$.
\end{thm}
%%%
Our second aim is to study the large time behavior of the global solution of \eqref{eq:1.1}.
%%%
\begin{thm} \label{Thm:1.2}
Under the same assumptions as in Theorem \ref{Thm:1.1},  
the global solution $u$ satisfies the following estimates; 
\begin{equation}
\begin{split}
& \lim_{t \to \infty} t^{\frac{n}{4}+\frac{k}{2}}\| |\nabla_{x}|^{k} (u(t)-MG_{t}) \|_{L^{2}(\R^{n})} =0,
\end{split}
\end{equation}
for $k \in [0, k_{0}]$,
where $M:= \displaystyle{\int_{\R^{n}}} (u_{0}(y) +u_{1}(y)) dy + \displaystyle{\int_{0}^{\infty} \int_{\R^{n}}} f(u(s,y)) dy ds$.
\end{thm}
Finally, we show the non-existence of global solution.  
\begin{thm} \label{Thm:1.3}
Let $n \ge 1$ and $1 < p \le 1+\displaystyle{\frac{2}{n}}$.
Assume that
$(u_{0}, u_{1}) \in W^{2.1} \cap W^{2, \infty} \times L^{1} \cap L^{\infty}$ 
satisfying 
\begin{equation}
\int_{\R^{n}} u_{i}(x) > 0
\end{equation}
for $i=0,1$.
Then the global solution of \eqref{eq:1.1} does not exist.
\end{thm}
\begin{rem}{\rm As consequences of Theorems 1.1 and 1.3, at least $n = 1,2$ cases we can find that the critical exponent to problem \eqref{eq:1.1} is $p_{F} = 1+\displaystyle{\frac{2}{n}}$. In our forthcoming paper we will soon announce the sharp result on the $n = 3$ dimensional case. It is still open to show similar results for all $n \geq 4$.}
\end{rem}

Before closing this section, 
we summarize the notation, which are used throughout this paper.
Let $\hat{f}$ denote the Fourier transform of $f$
defined by
\begin{align*}
\hat{f}(\xi) := c_{n}
\int_{\R^{n}} e^{-i x \cdot \xi} f(x) dx
\end{align*}
with $c_{n}= (2 \pi)^{-\frac{n}{2}}$.
Also, let $\mathcal{F}^{-1}[f]$ or $\check{f}$ denote the inverse
Fourier transform.

We introduce smooth cut-off functions to localize the frequency as follows: 
$\chi_L $, $\chi_M$ and $\chi_H  \in C^{\infty}(\mathbb{R})$ are defined by 
\begin{gather*}
\chi_L (\xi) = \begin{cases}
	1, \quad &|\xi| \leq \displaystyle{\frac{1}{2}}, \\
	0, \quad &|\xi| \geq \displaystyle{\frac{3}{4}}, 
	\end{cases} \qquad
\chi_H (\xi) = \begin{cases}
	1, \quad &|\xi| \geq 3, \\
	0, \quad &|\xi| \leq 2, 
	\end{cases} \\ 
\chi_M (\xi) = 1- \chi_L (\xi) - \chi_H (\xi). 
\end{gather*}

For $k \ge 0$, let  $H^{k}(\R^{n})$ be the Sobolev space;
\begin{equation*}
H^{k}(\R^{n})
  :=\Big\{ f:\R^{n} \to \R;
        \| f \|_{H^{k}(\R^{n})} 
        := (\Vert f \Vert_{2}^{2} + 
         \Vert |\nabla_{x}|^{k} f \Vert_{2}^{2})^{1/2}< \infty 
     \Big\},
\end{equation*}
where $L^{p}(\R^{n})$ is the Lebesgue space for $1 \le p \le \infty$ as usual.
For the notation of the function spaces, 
the domain $\R^{n}$ is often abbreviated, and we frequently use the notation $\| f \|_{p} =\| f \|_{L^{p}(\R^{n})}$ without confusion.
We write $B_{r}(0)$ for the $n$ dimensional open ball centered at the origin with the radius $r>0$;
$$
B_{r}(0) := \{ x \in \R^{n} ; |x|<r \}. 
$$
In the following, $C$ denotes a positive constant. 
It may change from line to line.

The paper is organized as follows. 
Section 2 presents some preliminaries.
In Section 3,
we show an appropriate decomposition of the propagators for the linear equation in the Fourier space. 
Section 4 is devoted to the proof of linear estimates, which play a crucial role in the main results.
In Sections 5, we construct the global solution for small initial data under the condition $p>1+\displaystyle{\frac{2}{n}}$. 
Section 6 provides the large time behavior of the global solution to prove Theorem 1.2.
In Section 7 we deal with the case $1<p \le 1+\displaystyle{\frac{2}{n}}$ in order to prove the non-existence result of global solutions.

%%%%%%%%%%%%%%%%%%%%%%%%%%%%%%%%%%%%%%%%%%%%%%%%%Section 2%%%%%%%%%%%%%%%%%%%%%%%%%%%%%%%%%%%%%%%%%%%%%%%%%%%%%%%%%%%%%%%%%%%%%%%%%%%%%%%%%%%%%%%%%%%%%%%%%%

\section{Preliminaries}
In this section, 
we shall recall useful estimates to show the results in this paper.
The following well-known estimate is frequently used to obtain time decay estimates.  
\begin{lem} \label{Lem:2.1}
Let $n \ge 1$, $k \ge 0$ and $1 \le r \le 2$. 
Then there exists a constant $C>0$ such that  
\begin{equation} \label{eq:2.1}
\| |\xi|^{k} e^{-(1+t)|\xi|^{2}} \|_{r} 
\le C (1+ t)^{-\frac{n}{2r}- \frac{k}{2}}.
\end{equation}
\end{lem}
\begin{proof}
The direct calculation immediately shows
\begin{equation*}
\| |\xi|^{k} e^{-(1+t)|\xi|^{2}} \|_{r}^{r}
= 
\int_{\R^{n}}  |\xi|^{kr} e^{-(1+t)r|\xi|^{2}} d \xi 
\le C \int_{0}^{\infty}  \tau^{kr+n-1} e^{-(1+t)r\tau^{2}} d \tau,
\end{equation*}
since the integrand in the left hand side is radial. 
Here changing the integral variable $\eta = \sqrt{r(1+t)} \tau$, 
we obtain 
\begin{equation*}
\int_{0}^{\infty}  \tau^{kr+n-1} e^{-(1+t)r\tau^{2}} d \tau
= 
C (1+t)^{-\frac{n}{2}-\frac{kr}{2}} \int_{0}^{\infty}  \eta^{kr+n-1} e^{-\eta^{2}} d \tau
\le C(1+t)^{-\frac{n}{2}-\frac{kr}{2}},
\end{equation*}
which is the desired estimate.
We complete the proof of Lemma \ref{Lem:2.1}.
\end{proof}
The following lemma is an easy consequence of the H\"older inequality.
%%%%
%%%
\begin{lem} \label{Lem:2.2}
Let $n \ge 1$, $1 \le r \le 2$ and $\displaystyle{\frac{1}{r}} + \displaystyle{\frac{1}{r'}} = 1$.
Then it holds that 
\begin{equation} \label{eq:2.2}
\| f g \|_{2} \le \| f \|_{\frac{2r}{2-r}} \| g \|_{r'}.
\end{equation}
\end{lem}
\begin{proof}
Noting that
\begin{equation*}
\frac{1}{\frac{r'}{2}} + \frac{2-r}{r}=1,
\end{equation*}
we simply apply the H\"older inequality to have 
\begin{equation*}
\begin{split}
\| f g \|_{2}^{2} & = \int_{\mathbb{R}^{n}} f^{2} g^{2} dx  \le \| f^{2} \|_{\frac{r}{2-r}} \| g^{2} \|_{\frac{r'}{2}}
 = \| f \|_{\frac{2r}{2-r}}^{2} \| g \|_{r'}^{2}. 
\end{split}
\end{equation*}
We complete the proof of Lemma \ref{Lem:2.2}.
\end{proof}
The following Lemma is the well-known Sobolev inequality (See e.g.\cite{C}).
\begin{lem} \label{Lem:2.3}
Let $n=1, 2$ and $\varepsilon>0$.
Then there exists a constant $C>0$ such that 
\begin{align} \label{eq:2.3}
\| u \|_{\infty} \le 
\begin{cases}
& C \| u \|_{2}^{\frac{1}{2}} 
 \| \nabla u \|_{2}^{\frac{1}{2}} \ \text{for}\ n=1, \\
\ \\
& C \| u \|_{2}^{\frac{\varepsilon}{1+\varepsilon}} 
 \| |\nabla|^{1+ \varepsilon} u \|_{2}^{\frac{1}{1+\varepsilon}}  \ \text{for}\ n=2.
\end{cases}  
\end{align} 
\end{lem}

The next lemma is useful to compute the decay order 
of the nonlinear term in the integral equation.

\begin{lem} \label{Lem:2.4}
%Given any $t \geq 0$, the following properties hold: \\ 
{\rm (i)} Let $a>0$ and 
$b>0$ with $\max\{ a, b \} >1$. Then, there exists a constant $C$ depending only on $a$ and $b$ such that
\begin{equation}
\int_{0}^{t} (1+t-s)^{-a} (1+s)^{-b} ds
\le 
C(1+t)^{-\min\{a, b \}}.  \label{eq:2.4}\\
\end{equation}
{\rm (ii)} Let $1 > a \geq 0$, $b > 0$ and $c>0$. Then, there exists a constant $C$ which is independent of $t$ such that  
\begin{equation}
\int_{0}^{t} e^{-c(t-s)} (t-s)^{- a}(1+s)^{- b} ds
\le C(1+t)^{-b}.  \label{eq:2.5}
\end{equation} 
\end{lem}

The proof of Lemma \ref{Lem:2.4} is well-known, and we omit its proof (See e.g. \cite{S}).
%%%%%%%%%%%%%%%%%%%%%%%%%%%%%%%%%%%%%%%%%%%%%%%%%%%%%%%%%%%%%%%%%%%%%%%%%%%%%%%%%%%%%%%%%%%%%%%%%%%%%%%%%%%%%%%%%%%%%%%%%%%%%%%%%%%%%%%%%%%%%%%

%%%%
The following Lemma is also well-known as 
the decay property and
approximation formula 
of the solution of the heat equation. 
For the proof, see e.g. \cite{GGS}.
\begin{lem}
Let $n \ge 1$, $\ell \ge 0$, $k \ge \tilde{k} \ge 0$ and $1 \le r \le 2$.
Then it holds that
\begin{equation} \label{eq:2.6}
\|\partial_{t}^{\ell} \nabla_{x}^{k} e^{t \Delta} g  \|_{2} 
\le C t^{-\frac{n}{2}(\frac{1}{r} -\frac{1}{2})-\ell-\frac{k-\tilde{k}}{2}} 
\| \nabla_{x}^{k-\tilde{k}} g \|_{r}
\end{equation} 
and 
\begin{equation} \label{eq:2.7}
\lim_{t \to \infty}
t^{\frac{n}{4}+\frac{k}{2}} 
\| \nabla_{x}^{k} (e^{t \Delta} g -m G_{t}) \|_{2}=0,
\end{equation} 
where $m= \displaystyle{\int_{\R^{n}}}g(y) dy$.
\end{lem} 
%

%%%%%%%%%%%%%%%%%%%%%%%%%%%%%%%%%%%%%%%%%%section 3%%%%%%%%%%%%%%%%%%%%%%%%%%%%%%%%%%%%%%%%%%%%%%%%%%%%%%%%%%%%%%%%%%%%%%%%%%%%%%%%%%%%%%%%%%%%%%%%%%%%%%

\section{Fourier multiplier expression} 
In this section, for the reader's convenience we repeat the derivation of the evolution operator of the linear problem.
We note that the argument here is already pointed out by Ikehata-Sawada \cite{IS}.

%%%%%%%%%%%%%%%%%%%%%%%%%%%%%%
%%%%%%%%%%%%%%%%%%%%%%%%%%%%%%%
%%%%%%%%%%%%%%%%%%%%%%%%%%%%%%%%%

Applying Fourier transform to \eqref{eq:1.1} with $f=0$, 
we have 
\begin{equation} \label{eq:3.1}
\left\{
\begin{split}
& \partial_{t}^{2} \hat{u} +(1+|\xi|^{2}) \partial_{t} \hat{u} +|\xi|^{2} \hat{u} =0, \quad t>0, \quad \xi \in \R^{n}, \\
& \hat{u}(0,\xi)=\hat{u}_{0}(\xi), \quad \partial_{t} \hat{u}(0,\xi)=\hat{u}_{1}(\xi) , \quad \xi \in \R^{n}.
\end{split}
\right.
\end{equation}
Then we see that
the characteristic equation of \eqref{eq:3.1}
is given by
\begin{equation*} 
\lambda^{2} + (1+|\xi|^{2}) \lambda + |\xi|^{2}=0,
\end{equation*}
and the characteristic roots of \eqref{eq:3.1} is
\begin{equation*}
\lambda_{\pm}
= \frac{-|\xi|^{2} \pm |1-|\xi|^{2}| }{2}.
\end{equation*}
In other words,
\begin{equation*}
\lambda_{+}
=
\begin{cases}
& -|\xi|^{2}  \quad  (|\xi| \le 1), \\
&  -1 \quad (|\xi| \ge 1),
\end{cases}
\quad
\lambda_{-}
=
\begin{cases}
& -1  \quad  (|\xi| \le 1), \\
&  -|\xi|^{2} \quad (|\xi| \ge 1).
\end{cases}
\end{equation*}
Therefore one can write the solution of \eqref{eq:3.1} explicitly by using the constants $C_{1}$ and $C_{2}$ such that 
\begin{equation} \label{eq:3.2}
\hat{u}(t) = C_{1} e^{\lambda_{+}t } + C_{2} e^{\lambda_{-}t }.
\end{equation}
This leads 
\begin{align*}
& C_{1} + C_{2} =\hat{u}_{0}, \\
& C_{1} \lambda_{+}  + C_{2} \lambda_{-} = \hat{u}_{1},
\end{align*}
namely, we have 
\begin{equation*}
C_{1}= 
\frac{-\lambda_{-} \hat{u}_{0}-\hat{u}_{1}}{\lambda_{+} - \lambda_{-}}, 
\qquad 
C_{2}= 
\frac{\lambda_{+} \hat{u}_{0}+\hat{u}_{1}}{\lambda_{+} - \lambda_{-}}, 
\end{equation*}
so that 
\begin{equation*}
\hat{u}(t) 
= \frac{-\lambda_{-} e^{\lambda_{+}t } + -\lambda_{+}  e^{\lambda_{-}t } }{\lambda_{+} - \lambda_{-}} \hat{u}_{0}
+ 
 \frac{-e^{\lambda_{-}t } + e^{\lambda_{+}t } }{\lambda_{+} - \lambda_{-}} \hat{u}_{1}.
\end{equation*}
Here we define $\K_{0}(t, \xi)$ and $\K_{1}(t,\xi)$ as 
\begin{equation*}
\begin{split}
\K_{0}(t, \xi) 
& := 
\frac{-\lambda_{-} e^{\lambda_{+}t } + \lambda_{+}e^{\lambda_{-}t } }{\lambda_{+} - \lambda_{-}}
= 
\dfrac{e^{-t|\xi|^{2}} -|\xi|^{2} e^{-t}}{1-|\xi|^{2}} , \\
\K_{1}(t, \xi) 
& := 
\frac{-e^{\lambda_{-}t } + e^{\lambda_{+}t } }{\lambda_{+} - \lambda_{-}} = \dfrac{e^{-t|\xi|^{2}} -e^{-t}}{1-|\xi|^{2}},  
\end{split}
\end{equation*}
and the evolution operators $K_{0}(t)g$ and $K_{1}(t) g$ as 
\begin{equation} \label{eq:3.3}
\begin{split}
K_{j}(t)g 
:= \mathcal{F}^{-1} [\K_{j}(t) \hat{g}]
\end{split}
\end{equation}
for $j=0,1$.
%%%%%%%%%%%%%%%%%%%%%%%%%%
%%%%%%%%%%%%%%%%%%%%%%%%%%%

%
\section{Linear estimates}
In this section, we consider the Cauchy problem of the linear equation 
\begin{equation} \label{eq:4.1}
\left\{
\begin{split}
& \partial_{t}^{2} u -\Delta u + \partial_{t} u -\Delta \partial_{t} u =0, \quad t>0, \quad x \in \R^{n}, \\
& u(0,x)=u_{0}(x), \quad \partial_{t} u(0,x)=u_{1}(x) , \quad x \in \R^{n}.
\end{split}
\right.
\end{equation}
We note that the results in this section are valid for all $n \ge 1$. 
Our aim is to show the following proposition, 
which means the decay properties of the problem \eqref{eq:4.1} and its asymptotic behavior. 
\begin{prop} \label{Prop:4.1}
Let $n \ge 1$ and $k_{0} \ge 0$.
Assume that 
$(u_{0}, u_{1}) \in L^{1} \cap H^{k_{0}}\times L^{1} \cap L^{2}$. 
Then, there exists unique solution of \eqref{eq:4.1} in the class $C([0, \infty);H^{k_{0}})$ 
satisfying
\begin{align} \label{eq:4.2}
& \| |\nabla_{x}|^{k} u(t) \|_{L^{2}(\R^{n})} \le C(1+t)^{-\frac{n}{4}-\frac{k}{2}}
\end{align}
for $k \in [0, k_{0}]$ and 
\begin{align} \label{eq:4.3}
&  \| |\nabla_{x}|^{k} (u(t)-MG_{t}) \|_{L^{2}(\R^{n})} =o(t^{-\frac{n}{4}-\frac{k}{2}})
\end{align}
as $t \to \infty$ for $k \in [0, k_{0}]$.
\end{prop}
%

%%%%%%%%%%%%%%%%%%%%%%%%%%%%%%%%%%%%%%%%%%%%%%%%%%%%%%%%%%%%%%subsection 4.1%%%%%%%%%%%%%%%%%%%%%%%%%%%%%%%%%%%%%%%%%%%%%%%%%%%%%%%%%%%%%%%%%%%%%%%%%%%%%%

\subsection{Decay estimates for evolution operators}
In this subsection, 
we show the decay properties of the evolution operators \eqref{eq:3.3}.  
\begin{lem} \label{Lem:4.2}
Let $n \ge 1$, $1 \le r \le 2$ and $0 \le \tilde{k} \le k$.
Then there exists a constant $C>0$ such that 
\begin{equation} \label{eq:4.4}
\| |\nabla_{x}|^{k} K_{0}(t) g \|_{2}
\le
C(1+t)^{-\frac{n}{2}(\frac{1}{r}-\frac{1}{2})-\frac{k-\tilde{k}}{2} }
\| |\nabla_{x}|^{\tilde{k}} g \|_{r} 
+ C e^{-t } \| |\nabla_{x}|^{k}  g \|_{2}.
\end{equation}
\end{lem} 
\begin{proof}
By the Planchrel formula, 
we see that 
\begin{equation*}
\| |\nabla_{x}|^{k} K_{0}(t) g \|_{2}
= 
\left\| |\xi|^{k} \K_{0}(t, \xi) \sum_{j=L,M,H} \chi_{j} \hat{g} \right\|_{2}
\le \sum_{j=L,M,H} \left\| |\xi|^{k} \K_{0}(t, \xi) \chi_{j} \hat{g} \right\|_{2}.
\end{equation*}
Now we estimate $\left\| |\xi|^{k} \K_{0}(t, \xi) \chi_{j} \hat{g} \right\|_{2}$ for $j=L,M,H$,
respectively.
Firstly we treat the case $j=L$.
Lemmas \ref{Lem:2.1} and \ref{Lem:2.2} imply that 
\begin{equation} \label{eq:4.5}
\begin{split}
\left\| |\xi|^{k} \K_{0}(t, \xi) \chi_{L} \hat{g} \right\|_{2}
& \le C \left\| |\xi|^{k} e^{-(1+t)|\xi|^{2}} \chi_{L} \hat{g} \right\|_{2}
+  C  e^{-t} \left\|  |\xi|^{\tilde{k}} \chi_{L} \hat{g} \right\|_{2} \\
& \le C \left\| |\xi|^{k-\tilde{k}} e^{-(1+t)|\xi|^{2}} \chi_{L} \right\|_{\frac{2r}{2-r}}
\left\| |\xi|^{\tilde{k}} \chi_{L} \hat{g} \right\|_{r'}
+  C  e^{-t} \left\| |\xi|^{\tilde{k}} \chi_{L} \hat{g} \right\|_{r'} \\
& \le C(1+t)^{-\frac{n}{2}(\frac{1}{r}-\frac{1}{2})-\frac{k-\tilde{k}}{2} }
\| |\nabla_{x}|^{\tilde{k}} g \|_{r}. 
\end{split}
\end{equation}
When $j=M$, 
the support of the middle part 
$\left\| |\xi|^{k} \K_{0}(t, \xi) \chi_{M} \hat{g} \right\|_{2}$
is compact and doesn't contain neighborhood of the origin $\xi=0$.
Then the integrand 
$ |\xi|^{k} \K_{0}(t, \xi) \chi_{M} \hat{g}$
has sufficient regularity and decays exponentially in $t$.
Namely we easily see that 
\begin{equation} \label{eq:4.6}
\left\| |\xi|^{k} \K_{0}(t, \xi) \chi_{M} \hat{g} \right\|_{2}
\le C e^{-\frac{t}{8}} \| |\nabla_{x}|^{\tilde{k}} g \|_{r}
\end{equation}
by the same way as the estimate \eqref{eq:4.5}.

For $j=H$, 
we can easily have 
\begin{equation} \label{eq:4.7}
\begin{split}
\left\| |\xi|^{k} \K_{0}(t, \xi) \chi_{H} \hat{g} \right\|_{2}
& \le C \left\| |\xi|^{k-2} e^{-t|\xi|^{2}} \chi_{H} \hat{g} \right\|_{2}
+  C  e^{-t} \left\|  |\xi|^{k} \chi_{H} \hat{g} \right\|_{2} \\
& \le C  e^{-t} \left\|  |\xi|^{k} \chi_{H} \hat{g} \right\|_{2} 
\le Ce^{-t}
\| |\nabla_{x}|^{k} g \|_{2}. 
\end{split}
\end{equation}
where we used the fact that $\supp \chi_{H} \subset \{ |\xi| \ge 3 \}$.
Therefore by \eqref{eq:4.5}, \eqref{eq:4.6} and \eqref{eq:4.7}, 
we obtain the desired estimate \eqref{eq:4.1}.
We complete the proof of Lemma \ref{Lem:4.2}.
\end{proof}
\begin{lem} \label{Lem:4.3}
Let $n \ge 1$, $1 \le r \le 2$ and $k \ge \tilde{k} \ge 0$.
Then there exists a constant $C>0$ such that 
\begin{equation} \label{eq:4.8}
\| |\nabla_{x}|^{k} K_{1}(t) g \|_{2}
\le
C(1+t)^{-\frac{n}{2}(\frac{1}{r}-\frac{1}{2})-\frac{k-\tilde{k}}{2}}
\| |\nabla_{x}|^{\tilde{k}} g \|_{r} 
+ C e^{-t } \| |\nabla_{x}|^{(k-2)_{+}}  g \|_{2},
\end{equation}
where $(k-2)_{+}:= \max \{k-2,0 \}$.
\end{lem}
\begin{proof}
The proof is similar to the proof of Lemma \ref{Lem:4.2}.
Here we omit the proof.
\end{proof}
%%%%
%%%
%%
%%%
%%%%%%
%%%%%%%%%%%%%%%%%%%%%%%%%%%%%%%%%%%%%%subsection 4.2%%%%%%%%%%%%%%%%%%%%%%%%%%%%%%%%%%%%%%%%%%%%%%%%%%%%%%%%%%%%%%%%%%%%%%%%%%%%%%%%%%%%%%%%%%%%%%%%%%%%%%%%
%%%%
\subsection{Asymptotic behavior of the linear solution}
Here we show the approximation of the evolution operators by the solution of the heat equation.
\begin{lem} \label{Lem:4.4}
Let $n \ge 1$, $1 \le r \le 2$ and $k \ge \tilde{k} \ge 0$.
Then there exists a constant $C>0$ such that 
\begin{equation}  \label{eq:4.9}
\| |\nabla_{x}|^{k} (K_{0}(t) - e^{t \Delta}) g \|_{2}
\le
C(1+t)^{-\frac{n}{2}(\frac{1}{r}-\frac{1}{2})-\frac{k-\tilde{k}}{2}-1 }
\| \nabla_{x}^{\tilde{k}} g \|_{r} 
+ C e^{-t} \| |\nabla_{x}|^{k}  g \|_{2}.
\end{equation}
\end{lem} 
\begin{proof}
Observing that
\begin{equation} \label{eq:4.10}
\| |\nabla_{x}|^{k} (K_{0}(t) - e^{t \Delta}) g \|_{2}
=
\left\| |\xi|^{k} (\K_{0}(t, \xi)-e^{-t|\xi|^{2}}) \sum_{j=L,M,H} \chi_{j} \hat{g} \right\|_{2}
\end{equation}
by the Planchrel formula,
we split the integrand of the right hand side in \eqref{eq:4.10} as follows:
\begin{equation} \label{eq:4.11}
\begin{split}
& |\xi|^{k} (\K_{0}(t, \xi)-e^{-t|\xi|^{2}}) \sum_{j=L,M,H} \chi_{j} \hat{g} \\
& = 
 |\xi|^{k}  (\K_{0}(t, \xi)-e^{-t|\xi|^{2}}) \chi_{L} \hat{g} 
 +|\xi|^{k} \K_{0}(t, \xi)  \sum_{j=M,H} \chi_{j} \hat{g} 
 +|\xi|^{k} e^{-t|\xi|^{2}} \sum_{j=M,H} \chi_{j} \hat{g}. 
\end{split}
\end{equation}
For the first factor in \eqref{eq:4.11}, 
we see that 
\begin{equation} \label{eq:4.12}
\begin{split}
||\xi|^{k}  (\K_{0}(t, \xi)-e^{-t|\xi|^{2}}) \chi_{L} \hat{g} |
& \le 
\left(
\frac{e^{-t|\xi|^{2}} |\xi|^{2} }{1-|\xi|^{2}}
+ e^{-t} \frac{|\xi|^{2}}{1-|\xi|^{2}}\right) \chi_{L} \hat{g}  \\
& \le C e^{-(1+t)|\xi|^{2}} |\xi|^{2} \hat{g} ,
\end{split}
\end{equation}
where we have just used the fact $\supp \chi_{L} \subset \{ |\xi| \le 1 \}$.
The second factor is estimated as in the proof of Lemma \ref{Lem:4.2}.
To obtain the estimate for the third factor in the right hand side of \eqref{eq:4.11}, 
it is also easy to see that 
\begin{equation} \label{eq:4.13}
|\xi|^{k} \left|e^{-t|\xi|^{2}} \sum_{j=M,H} \chi_{j} \hat{g} \right|
\le C e^{-t} |\xi|^{k} |\hat{g}|.
\end{equation}
Then summing up \eqref{eq:4.12} and \eqref{eq:4.13} and taking the $L^{2}$ norm for each terms, 
the simple calculation show the desired estimate \eqref{eq:4.9},
and the lemma now follows.
\end{proof}
\begin{lem} \label{Lem:4.5}
Let $n \ge 1$, $1 \le r \le 2$ and $k \ge \tilde{k} \ge 0$.
Then there exists a constant $C>0$ such that 
\begin{equation} \label{eq:4.14}
\| |\nabla_{x}|^{k} (K_{1}(t)-e^{t \Delta}) g \|_{2}
\le
C(1+t)^{-\frac{n}{2}(\frac{1}{r}-\frac{1}{2})-\frac{k-\tilde{k}}{2}-1}
\| |\nabla_{x}|^{\tilde{k}} g \|_{r} 
+ C e^{-t} \| |\nabla_{x}|^{(k-2)_{+}}  g \|_{2},
\end{equation}
where $(k-2)_{+}:= \max \{k-2,0 \}$.
\end{lem} 
\begin{proof}
By the same way to the proof of Lemma \ref{Lem:4.4}, 
we have Lemma \ref{Lem:4.5}.
\end{proof}
%
%\begin{rem}
%Under the notation of the above estimate \eqref{eq:4.14}
%We easily see that 
%\begin{equation} \label{eq:4.15}
%
%\| \nabla_{x}^{k} (K_{1}(t)-e^{t \Delta}) g \|_{2}
%\le
%C(1+t)^{-\frac{n}{2}(\frac{1}{r}-\frac{1}{2})-\frac{k-\tilde{k}}{2}-1}
%\| \nabla_{x}^{\tilde{k}} g \|_{r} 
%+ C e^{-t } \| |\nabla_{x}|^{\ell}  g \|_{2}.
%
%\end{equation}
%for $\ell \ge k$.
%\end{rem}
%%%%
%%%%%%%%%%%%%%%%%%%%%%%%%%%%%%%%%%%%%%%%%%%%%%%%%%%%%%%%%%\subsection4.3%%%%%%%%%%%%%%%%%%%%%%%%%%%%%%%%%%%%%%%%%%%%%%%%%%%%%%%%%%%%%%%%%%%%%%%%%%%%%%%%%%%%%%%%%%%%%%
%
\subsection{Proof of Proposition \ref{Prop:4.1}}
%%%%
%%%%%
%%%%%%
\begin{proof}
We recall that the solution to \eqref{eq:4.1} is expressed as 
$
u(t) = K_{0}(t)u_{0} +K_{1}(t)u_{1}.
$
Then it is clear that 
\begin{equation*}
\begin{split}
 \| |\nabla_{x}|^{k} u(t) \|_{L^{2}(\R^{n})}
\le \| |\nabla_{x}|^{k} K_{0}(t) u_{0} \|_{L^{2}(\R^{n})}
+ \| |\nabla_{x}|^{k} K_{1}(t) u_{1} \|_{L^{2}(\R^{n})}
\le C(1+t)^{-\frac{n}{4}-\frac{k}{2}}
\end{split}
\end{equation*}
using the estimates \eqref{eq:4.4} and \eqref{eq:4.8} with $\tilde{k}=0$ and $r=0$, 
which is the desired estimate \eqref{eq:4.2}.
To show the estimate \eqref{eq:4.3} is the combination of the estimates
\eqref{eq:4.9} and \eqref{eq:4.14} with $\tilde{k}=0$ and $r=0$ and \eqref{eq:2.7}.
That is,
\begin{equation*}
\begin{split}
& \| |\nabla_{x}|^{k} ( u(t) - MG_{t} )\|_{L^{2}(\R^{n})} \\
& \le \| |\nabla_{x}|^{k} (K_{0}(t)-e^{t \Delta}) u_{0} \|_{L^{2}(\R^{n})}
+ \| |\nabla_{x}|^{k} (K_{1}(t)-e^{t \Delta}) u_{1} \|_{L^{2}(\R^{n})} \\
& +\| |\nabla_{x}|^{k} (e^{t \Delta}(u_{0} +u_{1}) -MG_{t}) \|_{L^{2}(\R^{n})} \\
& \le C (1+t)^{-\frac{n}{4}-\frac{k}{2}-1} +o(t^{-\frac{n}{4}-\frac{k}{2}})
\end{split}
\end{equation*}
as $t \to \infty$, 
which is the desired estimate \eqref{eq:4.3}.
This proves Proposition \ref{Prop:4.1}.
\end{proof}
%%%%%%%%%%%%%%
%%%%%%%%%%%%%%%%%%%%%%%%%%%%%%%%%%%%%%%%%%%%%%%%%%%%%%%%%%%%%%%%%%%%%Section 5%%%%%%%%%%%%%%%%%%%%%%%%%%%%%%%%%%%%%%%%%%%%%%%%%%%%%%%%%%%%
%%%%%%%%%%%%%
\section{Existence of global solutions}
%%%%

This section is devoted to the proof of Theorem \ref{Thm:1.1}.

Here we prepare some notation.
For $n=1,2$, we define the closed subspace of $C([0,\infty); H^{k_{0}})$ as
\begin{equation*}
X_{n}:= \{ 
u \in C([0,\infty); H^{k_{0}}) ; \| u \|_{X_{n}} \le M_{n}
\},
\end{equation*}
where 
\begin{equation*}
\| u \|_{X_n} :=  \sup_{t \ge 0}
\{ 
(1+t)^{\frac{n}{4}} \| u(t) \|_{2} 
+ 
(1+t)^{\frac{n}{4}+ \frac{k_{0}}{2}}
\| |\nabla_{x}|^{k_{0}} u(t) \|_{2}
\}
\end{equation*}
and we determine $M_{n}>0$ later.
We also introduce the mapping on $X_{n}$ by 
\begin{equation} \label{eq:5.1}
\Phi[u](t)
:= 
K_{0}(t) u_{0} + K_{1}(t)u_{1} 
+ \int_{0}^{t} K_{1}(t-\tau) f(u)(\tau) d \tau.  
\end{equation}
For simplicity of the notation, 
we denote the integral term in \eqref{eq:5.1} as $I[u](t)$:
\begin{equation} \label{eq:5.2}
I[u](t) := \int_{0}^{t} K_{1}(t-\tau) f(u)(\tau) d \tau.
\end{equation}
In this situation, we claim that 
\begin{equation} \label{eq:5.3}
\| \Phi[u] \|_{X_{n}} \le M_{n} 
\end{equation}
for all $u \in X_{n}$ and 
\begin{equation} \label{eq:5.4}
\| \Phi[u]- \Phi[v] \|_{X_{n} }\le \frac{1}{2} \| u-v \|_{X_{n}}
\end{equation}
for $u,v \in X_{n}$.
For the proof of Theorem \ref{Thm:1.1}, it suffices to show the estimates \eqref{eq:5.3} and \eqref{eq:5.4}.
Indeed,
once we have obtained the estimates \eqref{eq:5.3} and \eqref{eq:5.4},
we see that $\Phi$ is a contraction mapping on $X_{n}$.
Therefore it is immediate from the Banach fixed point theorem that $\Phi$ has a unique fixed point in
$X_n$.
Namely, there exists a unique global solution $u$ of $u=\Phi[u]$ in $X_n$,
and Theorem \ref{Thm:1.1} is proved. 
%%%
We remark that the linear solution 
$K_{0}(t) u_{0} + K_{1}(t)u_{1}$ is estimated suitably by the linear estimates by Proposition \ref{Prop:4.1}.
In what follows we concentrate on the estimates for $I[u](t)$ defined by \eqref{eq:5.2}.
Firstly we prepare estimates of the norm for $f(u)$ and $f(u)-f(v)$, 
which used below:

Using the mean value theorem, 
we see that there exists $\theta \in [0,1]$ such that
\begin{equation*}
f(u) -f(v) = f'(\theta u + (1-\theta) v) (u-v).
\end{equation*}
Therefore 
noting that $\| u \|_{2(p-1)}^{p-1} \le \| u \|_{\infty}^{p-2} \| u \|_{2}$ since $2(p-1) \ge 2$
and Lemma \ref{Lem:2.3}, 
we arrive at the estimates
\begin{equation} \label{eq:5.5}
\begin{split}
\| f(u) - f(v) \|_{1} 
& \le \|  f'(\theta u + (1-\theta) v)\|_{2}
\| u-v \|_{2} \\
& \le C \| (\theta u + (1-\theta) v \|_{2(p-1)}^{p-1}
\| u-v \|_{2} \\
& \le C (\| u \|_{\infty}^{p-2}\| u \|_{2} + \| v \|_{\infty}^{p-2}\| v \|_{2})
\| u-v \|_{2} \\
& \le C (1+\tau)^{-\frac{n}{2}(p-1)}(\| u \|_{X_n}^{p-1} + \| v \|_{X_n}^{p-1})
\| u-v \|_{X_n} \\
& \le C (1+\tau)^{-\frac{n}{2}(p-1)} 2M_{n}^{p-1}
\| u-v \|_{X_n}.
\end{split}
\end{equation}
By the similar way, we have 
\begin{equation} \label{eq:5.6}
\begin{split}
\| f(u) - f(v) \|_{2} 
& \le C (\| u \|_{\infty}^{p-1} + \| v \|_{\infty}^{p-1})
\| u-v \|_{2} \\
& \le C (1+\tau)^{-\frac{n}{2}(p-1)-\frac{n}{4}}2M_{n}^{p-1}
\| u-v \|_{X_n} 
\end{split}
\end{equation}
and
\begin{equation} \label{eq:5.7}
\begin{split}
\||\nabla_{x}| (f(u) - f(v)) \|_{2} 
& = 
\|\nabla_{x} (f(u) - f(v)) \|_{2} \\
& = \| \nabla_{x} \{ f'(\theta u + (1-\theta) v) (u-v)  \} \|_{2} \\
& \le \|  f''(\theta u + (1-\theta) v) ( u-v) \nabla_{x} (\theta u + (1-\theta) v) \|_{2} \\
& + \| f'(\theta u + (1-\theta) v) \nabla_{x}(u-v)  \|_{2} \\
& \le C(\| u\|_{\infty}^{p-2} + \| v \|_{\infty}^{p-2} )  
\| u-v \|_{\infty} (\| \nabla_{x} u \|_{2} + \| \nabla_{x} v \|_{2}) \\
& + C(\| u\|_{\infty}^{p-1} + \| v \|_{\infty}^{p-1}) 
\| \nabla_{x} (u - v) \|_{2} \\
& \le C (1+\tau)^{-\frac{n}{2}(p-1)-\frac{n}{4}-\frac{1}{2}} 2M^{p-1}
\| u-v \|_{X_n}.
\end{split}
\end{equation}
When we take $v=0$ in \eqref{eq:5.5}-\eqref{eq:5.7} and recall $\| u \|_{X_n} \le M_n$, 
we easily see that 
\begin{equation} \label{eq:5.8}
\begin{split}
& \| f(u) \|_{1} \le C (1+\tau)^{-\frac{n}{2}(p-1)}M_n^{p},\\
& \| f(u) \|_{2} \le C (1+\tau)^{-\frac{n}{2}(p-1)-\frac{n}{4}}M_n^{p},\\
& \||\nabla_{x}| f(u) \|_{2} 
\le C (1+\tau)^{-\frac{n}{2}(p-1)-\frac{n}{4}-\frac{1}{2}} M_n^{p}.
\end{split}
\end{equation}
%%%%%%%%%%
Now, 
using the above estimates in \eqref{eq:5.8}, 
we show the estimate of $\| I[u](t) \|_{2}$ for $n=1,2$.
We apply the estimates \eqref{eq:4.8} with $k= \tilde{k}=0$ and $r=1$,
\eqref{eq:5.8},
\eqref{eq:2.4} and \eqref{eq:2.5} to have
\begin{equation} \label{eq:5.9}
\begin{split}
 \left\| I[u](t) \right\|_{2} 
& \le \int_{0}^{t} \left\| K_{1}(t- \tau) f(u) \right\|_{1}d \tau \\
& \le C \int_{0}^{t} (1+t -\tau)^{-\frac{n}{4}}\left\| f(u) \right\|_{1}d \tau 
+  C \int_{0}^{t} e^{-(t -\tau)}  \left\| f(u) \right\|_{2} d\tau \\
& \le C M_n^{p} \int_{0}^{t} (1+t -\tau)^{-\frac{n}{4}}  (1+\tau)^{-\frac{n}{2}(p-1)} d \tau \\
& \ +  C M_n^{p}  \int_{0}^{t} e^{-(t -\tau)} (1+\tau)^{-\frac{n}{2}(p-1)-\frac{n}{4}}d \tau \\
& \le  C(1+t)^{-\frac{n}{4}} M_n^{p}
\end{split}
\end{equation}
since $\displaystyle{\frac{n}{4}} \le 1$ for $n=1,2$.
%%%%%
Secondly 
by the similar way to the estimate \eqref{eq:5.9}, 
we calculate $\| I[u](t) -I[v](t) \|_{2}$ as follows.
\begin{equation} \label{eq:5.10}
\begin{split}
 \left\| I[u](t) -I[v](t) \right\|_{2} 
& \le \int_{0}^{t} \left\| K_{1}(t- \tau) (f(u) -f(v)) \right\|_{1}d \tau \\
& \le C \int_{0}^{t} (1+t -\tau)^{-\frac{n}{4}}\left\| f(u)-f(v) \right\|_{1}d \tau \\
& +  C \int_{0}^{t} e^{-(t -\tau)}  \left\| f(u)-f(v) \right\|_{2} d\tau \\
& \le C M_n^{p-1} \| u-v \|_{X_n} \int_{0}^{t} (1+t -\tau)^{-\frac{n}{4}}  (1+\tau)^{-\frac{n}{2}(p-1)} d \tau \\
& \ +  C M_n^{p-1}  \| u-v \|_{X_n} \int_{0}^{t} e^{-(t -\tau)} (1+\tau)^{-\frac{n}{2}(p-1)-\frac{n}{4}}d \tau \\
& \le  C (1+t)^{-\frac{n}{4}} M_n^{p-1} \| u-v \|_{X_n},
\end{split}
\end{equation}
where we have used \eqref{eq:5.5} and \eqref{eq:5.6}.
%%%%%%%
%%%%%%
The remainder part of the proof, we firstly show the case $n=1$.

%%%%%%%%%%%%%%%%%%%%%%%%%%%%%%%%%%%%%%%%subsection 5.1%%%%%%%%%%%%%%%%%%%%%%%%%%%%%%%%%%%%%%%%%%%%%%%%%%%%%%%%%%%%%%%%%%%%%%%%%%%%%%%%%%%%%%%%%%%%%%%%%%%%

\subsection{Proof of Theorem \ref{Thm:1.1} for $n=1$}
\begin{proof}
To prove Theorem \ref{Thm:1.1} for $n=1$,
we estimate $\| |\nabla| \Phi[u](t) \|_{2}$.
Using estimates \eqref{eq:4.8} with $k=1$, $\tilde{k}=1$ and $r=1$, 
\eqref{eq:5.8}, \eqref{eq:2.4} and \eqref{eq:2.5},
we obtain that
\begin{equation} \label{eq:5.11}
\begin{split}
 \left\| |\nabla_{x}| I[u](t) \right\|_{2} 
& \le \int_{0}^{t} \left\| |\nabla_{x}| K_{1}(t- \tau) f(u) \right\|_{2}d \tau \\
& \le C \int_{0}^{t}  (1+t-\tau)^{-\frac{1}{4}-\frac{1}{2}}\left\| f(u) \right\|_{1}d \tau 
+ C \int_{0}^{t}  e^{-(t- \tau)} \left\| f(u) \right\|_{2}d \tau \\
& \le C  \int_{0}^{t} (1+t-\tau)^{-\frac{1}{4}-\frac{1}{2}}(1+ \tau)^{-\frac{1}{2}(p-1)} d \tau M_1^{p} \\
& \ + C  \int_{0}^{t} e^{-(t-\tau)} (1+ \tau)^{-\frac{1}{2}(p-1)-\frac{1}{2}}  d \tau  M_1^{p} \\
& \le C  (1+t)^{-\frac{1}{4}-\frac{1}{2}} M_1^{p},
\end{split}
\end{equation}
where we have just used the fact that $-\displaystyle{\frac{1}{2}}(p-1)<-1$.

Next, we estimate $\| |\nabla_{x}|( \Phi[u](t)-\Phi[v](t) ) \|_{2}$.
We apply the argument of the estimate \eqref{eq:5.11} again, 
with \eqref{eq:5.8} replaced by \eqref{eq:5.5} and \eqref{eq:5.6} to obtain 
\begin{equation} \label{eq:5.12}
\begin{split}
\left\| |\nabla_{x}| (I[u](t)-I[v](t) )\right\|_{2} 
& \le \int_{0}^{t} \left\| |\nabla_{x}| K_{1}(t- \tau) (f(u)- f(v) ) \right\|_{2}d \tau \\
& \le C \int_{0}^{t} 
 (1+t-\tau)^{-\frac{1}{4}-\frac{1}{2}}\left\| f(u) -f(v) \right\|_{1}d \tau \\
& \ + C \int_{0}^{t}  e^{-(t- \tau)} \left\| f(u)-f(v) \right\|_{2}d \tau \\
& \le C  \int_{0}^{t} (1+t)^{-\frac{1}{4}-\frac{1}{2}}(1+ \tau)^{-\frac{1}{2}(p-1)} d \tau M_1^{p-1} 
\| u-v \|_{X_{1}} \\
& + C \int_{0}^{t} e^{-(t-\tau)} (1+ \tau)^{-\frac{1}{2}(p-1)-\frac{1}{2}}  d \tau M_1^{p-1} 
\| u-v \|_{X_{1}} \\
& \le C  (1+t)^{-\frac{1}{4}-\frac{1}{2}}  M_1^{p-1} 
\| u-v \|_{X_{1}}.
\end{split}
\end{equation}
By the estimates \eqref{eq:4.2} with $n=1$,
\eqref{eq:5.9} with $n=1$ and \eqref{eq:5.11},
we deduce that
\begin{equation} \label{eq:5.13}
\begin{split}
\| \Phi[u] \|_{X_1}
& \le \| K_{0}(t) u_0 +K_{1}(t)u_1 \|_{X_{1}}
+ \| I[u] \|_{X_1} \\
&
\le C_{0}(\| u_0 \|_{L^1\cap H^{k_0}}+\| u_1 \|_{L^1 \cap L^2}) +C_{1}M_{1}^{p}
\end{split}
\end{equation}
for some $C_0>0$ and $C_1>0$.

Similar arguments can be applied to the case $\| \Phi[u] -\Phi[v] \|_{X_1}$ 
using the estimates \eqref{eq:5.10} with $n=1$ and \eqref{eq:5.12},
and we can assert that
\begin{equation} \label{eq:5.14}
\| \Phi[u] -\Phi[v] \|_{X_1}
\le \| I[u] -I[v] \|_{X_{1}} 
\le C_2 M_{1}^{p-1} \| u-v \|_{X_{1}}
\end{equation}
for some $C_2>0$.
Then we choose $M_1=2 C_{0}(\| u_0 \|_{L^1\cap H^{k_0}}+\| u_1 \|_{L^1 \cap L^2})$
with sufficiently small $\| u_0 \|_{L^1\cap H^{k_0}}+\| u_1 \|_{L^1 \cap L^2}$ satisfying 
\begin{equation} \label{eq:5.15}
C_{1}M_{1}^{p} < \frac{1}{2} M_1,\quad
C_2 M_{1}^{p-1} < \frac{1}{2}.
\end{equation}
Combining the estimates \eqref{eq:5.13}, \eqref{eq:5.14} and \eqref{eq:5.15} yields 
the desired estimates \eqref{eq:5.3} and \eqref{eq:5.4}, and the proof of $n=1$ is now complete.
\end{proof}
Finally,
we show the remainder part of the proof for $n=2$. 
%%

%%%%%%%%%%%%%%%%%%%%%%%%%%%%%%%%%%%%%%%%%%%%subsection 5.2%%%%%%%%%%%%%%%%%%%%%%%%%%%%%%%%%%%%%%%%%%%%%%%%%%%%%%%%%%%%%%%%%%%%%%%%%%%%%%%%%%%%%%%%%%%%%%%%%%%

\subsection{Proof of Theorem \ref{Thm:1.1} for $n=2$}

%%%
For the proof of Theorem \ref{Thm:1.1} for $n=2$,
it remains to show the estimate of
$\| |\nabla_{x}|^{1+ \varepsilon} \Phi[u](t) \|_{2}$ with $n=2$.
Now,
we split the nonlinear term into two parts.
\begin{equation} \label{eq:5.16}
\begin{split}
 \left\| |\nabla_{x}|^{1+ \varepsilon} I[u](t) \right\|_{2} 
& \le \int_{0}^{\frac{t}{2}} \left\| |\nabla_{x}|^{1+ \varepsilon} K_{1}(t- \tau) f(u) \right\|_{2}d \tau \\
& \  +\int_{\frac{t}{2}}^{t}  \left\| |\nabla_{x}|^{\varepsilon} K_{1}(t- \tau)
  |\nabla_{x}| f(u) \right\|_{2} d \tau \\
& =:J_{1}(t) + J_{2}(t).
\end{split}
\end{equation}
To obtain the estimate of $J_{1}(t)$, we apply the estimates 
\eqref{eq:4.8} with $k=1+\varepsilon$, $\tilde{k}=0$, and $r=1$ and \eqref{eq:5.8}
to have
\begin{equation} \label{eq:5.17}
\begin{split}
J_{1}(t) 
& \le C \int_{0}^{\frac{t}{2}}  (1+t-\tau)^{-\frac{1}{2}-\frac{1+ \varepsilon}{2}}\left\| f(u) \right\|_{1}d \tau 
+ C \int_{0}^{\frac{t}{2}}  e^{-(t- \tau)} \left\| f(u) \right\|_{2}d \tau \\
& \le C  (1+t)^{-\frac{1}{2}-\frac{1+ \varepsilon}{2}} \int_{0}^{\frac{t}{2}} (1+ \tau)^{-(p-1)} d \tau M_2^{p}
+ C  e^{-\frac{1}{2} t}  \int_{0}^{\frac{t}{2}} (1+ \tau)^{-(p-1)-\frac{1}{2}}  \tau  M_2^{p} \\
& \le C  (1+t)^{-\frac{1}{2}-\frac{1+ \varepsilon}{2}} M_2^{p},
\end{split}
\end{equation}
where we have used the fact that $-(p-1)<-1$.
For the term $J_{2}(t)$,
using the estimates \eqref{eq:4.8} with  $k=1+\varepsilon$, $\tilde{k}=1$, and $r=2$ and 
\eqref{eq:5.8}, \eqref{eq:2.4} and \eqref{eq:2.5}, 
we obtain  
\begin{equation} \label{eq:5.18}
\begin{split}
J_{2}(t) 
& \le C \int_{\frac{t}{2}}^{t}  (1+t-\tau)^{-\frac{\varepsilon}{2}}\left\| |\nabla_{x}| f(u) \right\|_{2}d \tau 
+ C \int_{\frac{t}{2}}^{t}  e^{-(t- \tau)} \left\| f(u) \right\|_{2}d \tau \\
& \le C  
\int_{\frac{t}{2}}^{t} 
 (1+t-\tau)^{-\frac{\varepsilon}{2}} (1+\tau)^{-(p-1)-1} 
d \tau M_2^{p} \\
& + C  \int_{\frac{t}{2}}^{t} e^{-(t-\tau)}(1+ \tau)^{-(p-1)-\frac{1}{2}}  \tau  M_2^{p} \\
& \le C  (1+t)^{-(p-1)-\frac{\varepsilon}{2}} M_2^{p},
\end{split}
\end{equation}
where we remark that the power in the right hand side $-(p-1)-\frac{\varepsilon}{2}$ is strictly 
smaller than $-1-\frac{\varepsilon}{2}$, which have appeared in the estimate \eqref{eq:5.17}.
Combining the estimates \eqref{eq:5.16} - \eqref{eq:5.18}, 
we arrive at the estimate 
\begin{equation} \label{eq:5.19}
\begin{split}
\left\| |\nabla_{x}|^{1+ \varepsilon} I[u](t) \right\|_{2}  
\le 
J_{1}(t) +J_{2}(t)  \le C  (1+t)^{-\frac{1}{2}-\frac{1+ \varepsilon}{2}} M_2^{p}.
\end{split}
\end{equation}

Next, we estimate $\| |\nabla_{x}|^{1+ \varepsilon}( \Phi[u](t)-\Phi[v](t) ) \|_{2}$.
Again, we divide $\| |\nabla_{x}|^{1+ \varepsilon} (I[u](t)-I[v](t) ) \|_{2} $ into two parts.
\begin{equation} \label{eq:5.20}
\begin{split}
\left\| |\nabla_{x}|^{1+ \varepsilon} (I[u](t)-I[v](t) )\right\|_{2} 
& \le \int_{0}^{\frac{t}{2}} \left\| |\nabla_{x}|^{1+ \varepsilon} K_{1}(t- \tau) (f(u)- f(v) ) \right\|_{2}d \tau \\
& \  +\int_{\frac{t}{2}}^{t}  
\left\| |\nabla_{x}|^{\varepsilon} K_{1}(t- \tau)  |\nabla_{x}|(f(u)- f(v)) \right\|_{2} d \tau \\
& =: J_{3}(t) + J_{4}(t).
\end{split}
\end{equation}
As in the proof of \eqref{eq:5.17}, 
we deduce that
\begin{equation} \label{eq:5.21}
\begin{split}
J_{3}(t) 
& \le C \int_{0}^{\frac{t}{2}} 
 (1+t-\tau)^{-\frac{1}{2}-\frac{1+ \varepsilon}{2}}\left\| f(u) -f(v) \right\|_{1}d \tau \\
& \ + C \int_{0}^{\frac{t}{2}}  e^{-(t- \tau)} \left\| f(u)-f(v) \right\|_{2}d \tau \\
& \le C  (1+t)^{-\frac{1}{2}-\frac{1+ \varepsilon}{2}} \int_{0}^{\frac{t}{2}} (1+ \tau)^{-(p-1)} d \tau M_2^{p-1} 
\| u-v \|_{X_{2}} \\
& + C  e^{-\frac{1}{2} t}  \int_{0}^{\frac{t}{2}} (1+ \tau)^{-(p-1)-\frac{1}{2}}  d \tau M_2^{p-1} 
\| u-v \|_{X_{2}} \\
& \le C  (1+t)^{-\frac{1}{2}-\frac{1+ \varepsilon}{2}}  M_2^{p-1} 
\| u-v \|_{X_{2}},
\end{split}
\end{equation}
where we have used the fact that $-(p-1)<-1$ again.
In the same manner as the estimate \eqref{eq:5.18}, 
we can assert that
\begin{equation} \label{eq:5.22}
\begin{split}
J_{4}(t) 
& \le C \int_{\frac{t}{2}}^{t}  (1+t-\tau)^{-\frac{\varepsilon}{2}}\left\| |\nabla| (f(u) -f(v)) \right\|_{2}d \tau \\
& \ + C \int_{\frac{t}{2}}^{t}  e^{-(t- \tau)} \left\| ( f(u) -f(v)) \right\|_{2}d \tau \\
& \le C 
\int_{\frac{t}{2}}^{t}  (1+t-\tau)^{-\frac{\varepsilon}{2}} (1+ \tau)^{-(p-1)-1} d \tau 
M_2^{p-1} 
\| u-v \|_{X_{2}} \\
& \ + C  \int_{\frac{t}{2}}^{t} e^{-(t-\tau)} (1+ \tau)^{-(p-1)-\frac{1}{2}} d \tau  M_2^{p-1} 
\| u-v \|_{X_{2}} \\
& \le C  (1+t)^{-(p-1)-\frac{\varepsilon}{2}}M_2^{p-1} 
\| u-v \|_{X_{2}}.
\end{split}
\end{equation}
The estimates \eqref{eq:5.20} - \eqref{eq:5.22} yield 
\begin{equation} \label{eq:5.23}
\begin{split}
\left\| |\nabla|^{1+ \varepsilon} (I[u](t) -I[v](t)) \right\|_{2}  
\le 
J_{3}(t) +J_{4}(t)  \le C  (1+t)^{-\frac{1}{2}-\frac{1+ \varepsilon}{2}} M_2^{p-1} \| u-v \|_{X_{2}}.
\end{split}
\end{equation}
%%
%%%
%%
The rest of the proof is similar to the one of the case $n=1$.
By the estimates \eqref{eq:4.2} with $n=2$,
\eqref{eq:5.9} with $n=2$ and \eqref{eq:5.19},
we deduce that
\begin{equation} \label{eq:5.24}
\begin{split}
\| \Phi[u] \|_{X_2}
& \le \| K_{0}(t) u_0 +K_{1}(t)u_1 \|_{X_{1}}
+ \| I[u] \|_{X_2} \\
&
\le \tilde{C}_{0}(\| u_0 \|_{L^1\cap H^{k_0}}+\| u_1 \|_{L^1 \cap L^2}) +\tilde{C}_{1}M_{2}^{p}
\end{split}
\end{equation}
for some $\tilde{C}_{0}>0$ and $\tilde{C}_{1}>0$.

Similar arguments can be applied to the case $\| \Phi[u] -\Phi[v] \|_{X_2}$ 
using the estimates \eqref{eq:5.10} with $n=2$ and \eqref{eq:5.23},
and we can assert that
\begin{equation} \label{eq:5.25}
\| \Phi[u] -\Phi[v] \|_{X_2}
\le \| I[u] -I[v] \|_{X_{2}} 
\le \tilde{C}_2 M_{2}^{p-1} \| u-v \|_{X_{2}}.
\end{equation}
for some $\tilde{C}_2>0$.
Then we choose $M_2=2 \tilde{C}_{0}(\| u_0 \|_{L^1\cap H^{k_0}}+\| u_1 \|_{L^1 \cap L^2})$
with sufficiently small $\| u_0 \|_{L^1\cap H^{k_0}}+\| u_1 \|_{L^1 \cap L^2}$ satisfying 
\begin{equation} \label{eq:5.26}
\tilde{C}_{1}M_{2}^{p} < \frac{1}{2} M_2,\quad
\tilde{C}_2 M_{2}^{p-1} < \frac{1}{2}.
\end{equation}
Combining the estimates \eqref{eq:5.24}, \eqref{eq:5.25} and \eqref{eq:5.26} yields 
the desired estimates \eqref{eq:5.3} and \eqref{eq:5.4}, and the proof for $n=2$ is now complete.

%
%%%%%%%%%%%%%%%%%%
%%%%%%%%%%%%%%%
%%%%%%%%%%%%%%%%%%%
\section{Asymptotic behavior of the solution}
%%%%%%%%%%%%%%%%%%%%%
%%%%%%%%%%%%%%%%%%%%%%
%%%%%%%%%%%%%%%%%%%%%%%
In this section, we describe the proof of Theorem \ref{Thm:1.2}.
%%%%
%%
%%%
For the proof of Theorem \ref{Thm:1.2}, 
we prepare slightly general setting.
Here, 
we introduce the function $F=F(t,x)\in L^{1}(0,\infty;L^{1}(\R^{n}))$ satisfying 
\begin{align} 
& \| F(t) \|_{1} \le C (1+t)^{-\frac{n}{2}(p-1)}, \label{eq:6.1} \\
& \| F(t) \|_{2} \le C (1+t)^{-\frac{n}{2}(p-1)-\frac{n}{4}}, \label{eq:6.2} \\
& \||\nabla_{x}| F(t) \|_{2} 
\le C (1+t)^{-\frac{n}{2}(p-1)-\frac{n}{4}-\frac{1}{2}}, \label{eq:6.3}
\end{align}
where $p>1+\displaystyle{\frac{2}{n}}$.
We can now formulate our main purpose in this section.
%%%
\begin{prop} \label{Prop:6.1}
Let $n \ge 1$, $0 \le k \le 3$ and $p>1+\displaystyle{\frac{2}{n}}$.
Assume \eqref{eq:6.1} - \eqref{eq:6.3}.
Then it holds that 
\begin{equation} \label{eq:6.4}
\left\| 
|\nabla_{x}|^{k}
\left(
\int_{0}^{t} K_{1}(t-\tau) F(\tau) d \tau 
- \int_{0}^{\infty} \int_{\R^{n}}  
 F(\tau, y) dy d \tau\cdot G_{t}(x) 
\right)
\right\|_{2} = o(t^{-\frac{n}{4}-\frac{k}{2}})
\end{equation}
as $t \to \infty$.
\end{prop}
%%%%
As a first step of the proof of Proposition \ref{Prop:6.1}, 
we split the nonlinear term into five parts.
Namely, we see that
\begin{equation*}
\begin{split}
& \int_{0}^{t} K_{1}(t-\tau) F(\tau) d \tau 
- \int_{0}^{\infty} \int_{\R^{n}} F(\tau, y) dy d \tau \cdot G_{t}(x) \\
& 
= 
\int_{0}^{\frac{t}{2}}( K_{1}(t-\tau)-e^{(t-\tau)\Delta}) F(\tau) d \tau 
+ \int_{\frac{t}{2}}^{t} K_{1}(t-\tau) F(\tau) d \tau \\
& + \int_{0}^{\frac{t}{2}}(e^{(t-\tau)\Delta} -e^{ t\Delta}) F(\tau) d \tau
 + \int_{0}^{\frac{t}{2}}\left( e^{t\Delta} F(\tau) - \int_{\R^{n}}F(\tau, y) dy\cdot G_{t}(x) \right) d \tau \\
& - \int_{\frac{t}{2}}^{\infty} 
\int_{\R^{n}} F(\tau, y) dy d\tau\cdot G_{t}(x),
\end{split}
\end{equation*}
and we define each terms as follows:
\begin{equation*}
\begin{split}
A_{1}(t)
& := \int_{0}^{\frac{t}{2}}( K_{1}(t-\tau)-e^{(t-\tau)\Delta}) F(\tau) d \tau, \\
A_{2}(t) 
& := \int_{\frac{t}{2}}^{t} K_{1}(t-\tau) F(\tau) d \tau, \ 
A_{3}(t) := \int_{0}^{\frac{t}{2}}(e^{(t-\tau)\Delta} -e^{ t\Delta}) F(\tau) d \tau, \\
A_{4}(t) &
:= \int_{0}^{\frac{t}{2}}\left( e^{t\Delta} F(\tau) - \int_{\R^{n}}F(\tau, y) dy\cdot G_{t}(x) \right) d \tau, \\
A_{5}(t) &:=  - \int_{\frac{t}{2}}^{\infty} 
\int_{\R^{n}} F(\tau, y) dy d \tau \cdot G_{t}(x).
\end{split}
\end{equation*}
%%%%%%%%%%
%%%%%%%%%
In what follows, 
we estimate $A_{j}(t)$ for each $j=1, \cdots, 5$, respectively. 
%%%%%%%%%%
%%%%%%%%%%
%%%%%%%%
%%%%%%%
\begin{lem} \label{Lem:6.2}
Under the assumption as in Proposition \ref{Prop:6.1},
there exists some constant $C>0$ such that
\begin{equation} \label{eq:6.5}
\| |\nabla_{x}|^{k} A_{1}(t) \|_{2}
\le
C(1+t)^{-\frac{n}{4}-\frac{k}{2}-1},
\end{equation}
\begin{equation} \label{eq:6.6}
\| |\nabla_{x}|^{k} A_{2}(t) \|_{2} \le Ct^{-\frac{n}{4}-\frac{k}{2}-\frac{n}{2}(p-1)+1}.
\end{equation}
\end{lem}
\begin{proof}
First, we show the estimate \eqref{eq:6.5}.
By \eqref{eq:4.14} with $\tilde{k}=0$ and $r=1$, 
\eqref{eq:6.1} and \eqref{eq:6.3}, 
we see that 
\begin{equation*}
\begin{split}
\| |\nabla_{x}|^{k} A_{1}(t) \|_{2} 
& \le 
\int_{0}^{\frac{t}{2}} \|  |\nabla_{x}|^{k} ( K_{1}(t-\tau)-e^{(t-\tau)\Delta}) F(\tau) \|_{2} d \tau \\
& \le 
C \int_{0}^{\frac{t}{2}} (1+t- \tau)^{-\frac{n}{4}-\frac{k}{2}-1}  \| F(\tau) \|_{1} d \tau
+C \int_{0}^{\frac{t}{2}} e^{-(t-\tau)} \|\nabla F(\tau) \|_{2} d \tau \\
& \le C (1+t)^{-\frac{n}{4}-\frac{k}{2}-1} 
\int_{0}^{\frac{t}{2}} (1 + \tau)^{-\frac{n}{2}(p-1)} d \tau \\
& +C e^{-\frac{t}{2}} 
\int_{0}^{\frac{t}{2}}  (1 + \tau)^{-\frac{n}{2}(p-1)-\frac{n}{4}-\frac{1}{2}}  d \tau \\
& \le C (1+t)^{-\frac{n}{4}-\frac{k}{2}-1}, 
\end{split}
\end{equation*}
which is the desired estimate \eqref{eq:6.5}.
%%%%%%%
%%%%%%
%%%%%%%%%%
Next, we show the estimate \eqref{eq:6.6}.
By \eqref{eq:4.14} with $\tilde{k}=0$ and $r=2$, 
\eqref{eq:6.1} and \eqref{eq:6.3}, 
we see that 
\begin{equation*}
\begin{split}
\| |\nabla_{x}|^{k} A_{2}(t) \|_{2} 
& \le 
\int_{\frac{t}{2}}^{t} \|  |\nabla_{x}|^{k} ( K_{1}(t-\tau)-e^{(t-\tau)\Delta}) F(\tau) \|_{2} d \tau \\
& \le 
C \int_{\frac{t}{2}}^{t} (1+t- \tau)^{-\frac{k-1}{2}}  \|\nabla F(\tau) \|_{2} d \tau
+C \int_{\frac{t}{2}}^{t} e^{-(t-\tau)} \| \nabla F(\tau) \|_{2} d \tau \\
& \le C \int_{\frac{t}{2}}^{t}
 (1 + t-\tau)^{-\frac{k-1}{2}}
(1+\tau )^{-\frac{n}{4}-\frac{n}{2}(p-1)-\frac{1}{2}} 
 d \tau \\
& +C 
\int_{\frac{t}{2}}^{t} e^{-(t-\tau)} (1 + \tau)^{-\frac{n}{2}(p-1)-\frac{n}{4}-\frac{1}{2}}  d \tau \\
& \le C
\begin{cases} 
& (1+\tau )^{-\frac{n}{4}-\frac{k}{2}-\frac{n}{2}(p-1)+1} \ \text{for} \ 0 \le k <3, \\  
& (1+\tau )^{-\frac{n}{4}-\frac{n}{2}(p-1)-\frac{1}{2}} \log(2+t) \ \text{for} \ k=3,  
\end{cases}
\end{split}
\end{equation*}
which is the desired estimate \eqref{eq:6.5}.
We complete the proof of Lemma \ref{Lem:6.2}.
\end{proof}
%%%%%%
%%%%%%
%%%%%%%%%%%%
For the terms $A_{j}(t)$, $j=3,4,5$, we can obtain the suitable decay estimates under the slightly weaker condition in Proposition \ref{Prop:6.1}. 
%%%%%%%%%%%%%%%%%
%%%%%%%%%%%%%%%
\begin{lem} \label{Lem:6.3}
Let $n \ge 1$, $k \ge 0$ and $p>1+\displaystyle{\frac{2}{n}}$.
Then the following estimates holds: 
\begin{equation} \label{eq:6.7}
\| |\nabla_{x}|^{k} A_{3}(t) \|_{2} 
\le 
\begin{cases}
& C t^{-\frac{n}{4}-\frac{k}{2}-1} \log(2+t), \quad p \ge 1+\frac{4}{n}, \\ 
& C t^{-\frac{n}{4}-\frac{k}{2}-\frac{n}{2}(p-1)+1}, \quad 1+\frac{2}{n} <p < 1+\frac{4}{n},
\end{cases}
\end{equation}
\begin{equation} \label{eq:6.8}
\| |\nabla_{x}|^{k} A_{4}(t) \|_{2} = o(t^{-\frac{n}{4}-\frac{k}{2}}), \quad (t \to +\infty),
\end{equation}
and 
\begin{equation} \label{eq:6.9}
\begin{split}
\| |\nabla_{x}|^{k} A_{5}(t) \|_{2} 
\le C t^{-\frac{n}{4}-\frac{k}{2}-\frac{n}{2}(p-1)+1}.
\end{split}
\end{equation}
\end{lem}
\begin{rem}
Here we note that $-\displaystyle{\frac{n}{2}}(p-1)+2<1$, since we have just assumed $p>1+\displaystyle{\frac{2}{n}}$.
Therefore we have
\begin{equation*}
-\frac{n}{4}-\frac{k}{2}-1-\frac{n}{2}(p-1)+2 < -\frac{n}{4}-\frac{k}{2},
\end{equation*}
which means the desired estimate 
$t^{\frac{n}{4}+ \frac{k}{2}}  \| |\nabla_{x}|^{k} A_{j}(t) \|_{2} \to 0$ as $t \to \infty$ for $j=2,3,5$.
\end{rem}
\begin{proof}
At first, we prove the estimate \eqref{eq:6.7}.
Observing that there exists $\theta \in [0,1]$ such that 
\begin{equation*}
G_{t-\tau}(x-y) -G_{t}(x-y)
= 
(-\tau) \partial_{t} G_{t-\theta \tau}(x-y)
\end{equation*}
by the mean value theorem on $t$, 
we can apply the estimate \eqref{eq:2.6} with $\tilde{k}=0$, $\ell=1$ and $r=1$ to have
\begin{equation*}
\begin{split}
\| |\nabla_{x}|^{k} A_{3}(t) \|_{2} 
& \le 
\int_{0}^{\frac{t}{2}} \| |\nabla_{x}|^{k} (e^{(t-\tau)\Delta} -e^{ t\Delta}) F(\tau) \|_{2} d \tau \\
& = 
\int_{0}^{\frac{t}{2}} \tau \| |\nabla_{x}|^{k} \partial_{t} e^{(t-\theta \tau)\Delta} F(\tau) \|_{2} d \tau \\
& \le C  
\int_{0}^{\frac{t}{2}} \tau (t- \tau)^{-\frac{n}{4}-\frac{k}{2}-1}  \| F(\tau) \|_{1} d \tau \\
& \le C t^{-\frac{n}{4}-\frac{k}{2}-1} 
\int_{0}^{\frac{t}{2}} \tau (1 + \tau)^{-\frac{n}{2}(p-1)} d \tau \\
& \le 
\begin{cases}
& C t^{-\frac{n}{4}-\frac{k}{2}-1} \log(2+t), \quad p \ge 1+\frac{4}{n}, \\ 
& C t^{-\frac{n}{4}-\frac{k}{2}-\frac{n}{2}(p-1)+1}, \quad 1+\frac{2}{n} <p < 1+\frac{4}{n}, 
\end{cases}
\end{split}
\end{equation*}
and we have just proved the estimate \eqref{eq:6.7}.
Next, we prove the estimate \eqref{eq:6.8}.
To show the estimate for $A_{4}(t)$, 
we firstly divide the integrand into two parts:
\begin{equation} \label{eq:6.10}
\begin{split}
& e^{t \Delta} F(\tau, x) - \int_{\R^{n}}F(\tau,y) dy G_{t}(x) \\
& = \int_{|y| \le t^{\frac{1}{4}}} +\int_{|y| \ge t^{\frac{1}{4}}}
(G_{t}(x-y) -G_{t}(x)) F(\tau,y) dy 
=: A_{41}(t) +A_{42}(t).
\end{split}
\end{equation}
In what follows, we estimate $A_{41}(t)$ and $A_{42}(t)$, respectively.
For the estimate for $A_{41}(t)$,
we apply the mean value theorem again on $x$ to have 
\begin{equation*}
G_{t}(x-y) -G_{t}(x)
= 
(-y) \cdot \nabla_{x} G_{t}(x-\tilde{\theta} y)
\end{equation*}
for some $\tilde{\theta} \in [0,1]$, where $\cdot$ denotes the standard Euclid inner product.
Then we arrive at the estimate 
\begin{equation} \label{eq:6.11}
\begin{split}
\| |\nabla_{x}|^{k} A_{41}(t) \|_{2} 
& \le 
\int_{0}^{\frac{t}{2}} 
\int_{|y| \le t^{\frac{1}{4}}}
\left\|  |\nabla_{x}|^{k} 
(G_{t}(x-y) -G_{t}(x))
\right\|_{L^{2}_{x}}
 |F(\tau,y)| dy 
 d \tau \\
& = 
\int_{0}^{\frac{t}{2}} 
\int_{|y| \le t^{\frac{1}{4}}}
\left\| 
(-y) \cdot  |\nabla_{x}|^{k} \nabla_{x} G_{t}(x-\tilde{\theta} y)
\right\|_{L^{2}_{x}}
 |F(\tau,y)| dy
 d \tau \\
& \le C  t^{-\frac{n}{4}-\frac{k}{2}-\frac{1}{2} +\frac{1}{4}} 
\int_{0}^{\frac{t}{2}}  \| F(\tau) \|_{1} d \tau \\
& \le C t^{-\frac{n}{4}-\frac{k}{2}-\frac{1}{4}} 
\int_{0}^{\frac{t}{2}} (1 + \tau)^{-\frac{n}{2}(p-1)} d \tau 
 \le C t^{-\frac{n}{4}-\frac{k}{2}-\frac{1}{4}},
\end{split}
\end{equation}
by the direct calculation.
On the other hand, 
for the term $A_{42}(t)$, 
we recall the fact that $\displaystyle{\int_{0}^{\infty}\int_{\R^{n}}} |F(\tau, y)| dy d \tau < \infty$
implies 
\begin{equation*}
\lim_{t \to \infty}
\int_{0}^{\infty} 
\int_{|y| \ge t^{\frac{1}{4}}}
|F(\tau, y)| dy d \tau =0.
\end{equation*}
Thus we see that 
\begin{equation*}
\begin{split}
& \| |\nabla_{x}|^{k} A_{42}(t) \|_{2} \\
& \le 
\int_{0}^{\frac{t}{2}} 
\int_{|y| \ge t^{\frac{1}{4}}}
(\left\|  |\nabla_{x}|^{k} 
G_{t}(x-y) 
\right\|_{L^{2}_{x}}
+
\left\|  |\nabla_{x}|^{k} G_{t}(x)
\right\|_{L^{2}_{x}}
)
 |F(\tau,y)| dy 
 d \tau \\
& \le C t^{-\frac{n}{4}-\frac{k}{2}}
\int_{0}^{\infty} 
\int_{|y| \ge t^{\frac{1}{4}}}
 |F(\tau,y)| dy
 d \tau,
\end{split}
\end{equation*}
and then 
\begin{equation} \label{eq:6.12}
t^{\frac{n}{4}+\frac{k}{2}} \| |\nabla_{x}|^{k} A_{42}(t) \|_{2} \to 0
\end{equation}
as $t \to \infty$.
Therefore combining the estimates \eqref{eq:6.10} and \eqref{eq:6.12},
we have the desired estimate \eqref{eq:6.8}.
Finally, we show the estimate \eqref{eq:6.9}.
By the combination of \eqref{eq:6.1} and the direct calculation,
we get
\begin{equation*}
\begin{split}
\| |\nabla_{x}|^{k} A_{5}(t) \|_{2} 
& \le 
\int_{\frac{t}{2}}^{\infty} 
\|F(\tau)\|_{1} d \tau \cdot \|  |\nabla_{x}|^{k} G_{t}\|_{2} \\
& \le 
\int_{\frac{t}{2}}^{\infty} 
(1+\tau)^{-\frac{n}{2}(p-1)} d \tau\cdot  \|  |\nabla_{x}|^{k} G_{t}\|_{2}
\le C t^{-\frac{n}{4}-\frac{k}{2}-\frac{n}{2}(p-1)+1},
\end{split}
\end{equation*}
which is the desired estimate \eqref{eq:6.9}.
We complete the proof of Lemma \ref{Lem:6.3}.
%%%%
%%%%
\end{proof}
%
%%%%%%%%%
%%%%%%%%%%%%%%%%%%%%%%%%%%%%%
%%%%%%%%%%%%%%%%%%%%%%%%%%%%%%
%%%%%%%%%%%%%%%%%%%%%%%%%%%%%
%%%%%%%%%%%%%%%%
%%%%%%%%%%%%%%%%%%
%%%%%%%%%%%%%%%%%%
%%%%%
%%%%%%
%%%%%

%
%%%
\begin{proof}[Proof of Proposition \ref{Prop:6.1}]
For $0 \le k \le 3$, 
Lemmas \ref{Lem:6.2} and \ref{Lem:6.3} immediately yield the estimate \eqref{eq:6.4}.
Indeed, from the estimates \eqref{eq:6.5}-\eqref{eq:6.8}, it follows that 
\begin{equation*}
\begin{split}
& t^{\frac{n}{4}+\frac{k}{2}}
\left\| 
|\nabla|^{k}
\left(
\int_{0}^{t} K_{1}(t-\tau) F(\tau) d \tau 
- \int_{0}^{\infty} \int_{\R^{n}}  
 F(\tau, y) dy d \tau\cdot G_{t}(x) 
\right)
\right\|_{2} \\
& \le 
C
t^{\frac{n}{4}+\frac{k}{2}}
\sum_{j=1}^{5}
\| |\nabla|^{k} A_{j}(t) \|_{2}
\to 0
\end{split}
\end{equation*}
as $t \to \infty$, which is the desired conclusion.
\end{proof}
Now we are in a position to prove Theorem \ref{Thm:1.2}.
\begin{proof}[Proof of Theorem \ref{Thm:1.2}]
From the proof of Theorem \ref{Thm:1.1}, 
we see that the nonlinear term $f(u)$ satisfies the conditions \eqref{eq:6.1} - \eqref{eq:6.3}.
Then we can apply Proposition \ref{Prop:6.1} as $F(\tau, y) = f(u(\tau,y))$,
and the proof is now complete.
\end{proof}
\section{Proof of the finite time blow-up of solutions}
In this section, 
applying the method in \cite{Z} and \cite{T}, 
we show the nonexistence of any global solutions when $1 < p \le 1+\displaystyle{\frac{2}{n}}$, 
even if the initial data is small. 

Now, we follow the notation used in \cite{Z}.
Let $\phi= \phi(|x|) \in C^{\infty}([0, \infty))$ satisfy 
\begin{align*}
& 0 \le \phi \le 1, \ \phi(|x|) \equiv 1 \ \text{on}\ |x| \in \left[0, \frac{1}{2} \right],  
\quad \phi(|x|) \equiv 0\ \text{on} \ |x| \in [1, \infty), \\ 
& -C \le \phi'(|x|) \le 0, \quad |\phi''(|x|)| \le C, \quad
\frac{|\nabla \phi|^{2}}{\phi} \le C 
\end{align*}
for some constant $C>0$.
Also let
$\eta= \eta(t) \in C^{\infty}([0, \infty))$
satisfying 
\begin{align*}
& 0 \le \eta \le 1, \ \eta(t) \equiv 1 \ \text{on}\ t \in \left[0, \frac{1}{4} \right],  
\quad \eta(t) \equiv 0\ \text{on} \ t \in [1, \infty), \\ 
& -C \le \eta'(t) \le 0, \quad |\eta''(t)| \le C, 
\quad \frac{|\eta'|^{2}}{\eta} \le C 
\end{align*}
for some constant $C>0$.
Here we define $\phi_{R}=\phi_{R}(|x|)$ and $\eta_{R} = \eta_{R}(t)$ as 
\begin{align} \label{eq:7.1}
\phi_{R}=\phi \left(\frac{|x|}{R} \right), \quad \eta_{R} = \eta\left(\frac{t}{R^{2}} \right)
\end{align}
and the definition \eqref{eq:7.1} immediately yields useful estimates as follows:
\begin{align}
& 0 \le \phi_{R} \le 1, 
\ \phi_{R} \equiv 1 \ \text{on}\ |x| \in \left[0, \frac{R}{2} \right],  
\quad \phi_{R} \equiv 0\ \text{on} \ |x| \in [R, \infty), \label{eq:7.2} \\ 
& -\frac{C}{R} \le \phi_{R}' \le 0, \quad |\phi_{R}''| \le \frac{C}{R^{2}}, \label{eq:7.3} \\
& \frac{|\nabla \phi_{R}|^{2}}{\phi_{R}} \le \frac{C}{R^{2}}, \label{eq:7.4}\\ 
& 0 \le \eta_{R} \le 1, \ \eta_{R} \equiv 1 \ \text{on}\ t \in \left[0, \frac{R^{2}}{4} \right],  
\quad \eta_{R} \equiv 0\ \text{on} \ t \in [R^{2}, \infty), \label{eq:7.5}\\ 
&  |\eta'_{R}| \le \frac{C}{R^{2}}, \quad |\eta_{R}''| \le \frac{C}{R^{4}}, \label{eq:7.6}\\
& \frac{|\eta_{R}'|^{2}}{\eta_{R}} \le \frac{C}{R^{4}}. \label{eq:7.7} 
\end{align}
We also note that a direct calculation gives 
\begin{equation*}
-\Delta \phi_{R}^{p'}
= p' \phi_{R}^{p'-1}(-\Delta \phi_{R}) + p'(p'-1)\phi_{R}^{p'-2} |\nabla \phi_{R}|^{2}
\end{equation*}
where $\displaystyle{\frac{1}{p}}+\displaystyle{\frac{1}{p'}}=1$ and then we see
\begin{equation} \label{eq:7.8}
\begin{split}
|\Delta \phi_{R}^{p'}|
\le C \phi_{R}^{p'-1}|\Delta \phi_{R}| + C\phi_{R}^{p'-2} |\nabla \phi_{R}|^{2} \le \frac{C}{R^{2}} \phi_{R}^{p'-1}
\end{split}
\end{equation}
by the estimates \eqref{eq:7.3} and \eqref{eq:7.4}. 
Now we introduce the functional $K_{R}$ by 
\begin{equation} \label{eq:7.9}
K_{R}:= \int_{0}^{R^{2}} \int_{B_{R}(0)} 
|u|^{p} \phi_{R}^{p'} \eta_{R}^{p'} dx dt,
\end{equation}
where $R>0$.

At the rest of this section one divides the proof into three parts.
The first lemma is concerned with the lower bound of $K_{R}$.
The important fact is that the lower bound is independent of $R$.
\begin{lem} \label{Lem:7.1}
Let $u$ be the global solution of \eqref{eq:1.1}.
Then there exists a constant $R_{0}>0$ such that the following estimate holds good: 
\begin{align} \label{eq:7.10}
K_{R}[u] \ge  C_{0}  
\end{align}
for some constant $C_{0}>0$, which is independent of $R$ satisfying $R \ge R_{0}$. 
\end{lem}
\begin{proof}
This proof is based on an idea of \cite{T}.
Without loss of generality,
we assume that
$\displaystyle{\int_{\mathbb R^{n}}} u_{0}(x) dx >0$.
Then,
there exists $r_{0}>0$ such that $\int_{B_{R}(0)} u_{0}(x) dx >0$ for $R \ge r_{0}$.
Let $f(t) := \displaystyle{\int_{B_{r_{0}}(0)}} u(t,x) dx$.
Since $f(0)>0$,
there exists $T_{0}>0$ such that,
for any $t \in [0,
T_{0}]$,
we have $f(t) >0$.
On the other hand, 
it follows from the H\"older inequality 
and choosing $R \ge r_{0}$ with $B_{r_{0}}(0) \subset B_{\frac{R}{2}}(0)$ that
\begin{align*}
f(t)  \le \left(
\int_{B_{r_{0}}(0)} |u|^{p} dx
\right)^{\frac{1}{p}}
\left(
\int_{B_{r_{0}}(0)} dx
\right)^{\frac{1}{p'}} 
\le C \left(
\int_{B_{R}(0)} |u|^{p} \phi_{R}^{p'} \, dx
\right)^{\frac{1}{p}}
r_{0}^{\frac{n}{p'}}.
\end{align*}
Moreover, if we choose $R > 0$ so large such as
$\displaystyle{\frac{R^{2}}{4}} \ge T_{0}$, then this leads to
\begin{align*}
C r_{0}^{-\frac{n}{p'}} 
\int_{0}^{T_{0}} f(t)^{p} dt \le \int_{0}^{R^{2}} 
\int_{B_{R}(0)} |u|^{p} \phi_{R}^{p'} \eta_{R}^{p'}\,
 dx dt =K_{R}[u].
\end{align*}
Therefore, by setting
\begin{align*}
C_{0}:=  C 
r_{0}^{-\frac{pn}{p'}} 
\int_{0}^{T_{0}} f(t)^{p} dt,
\end{align*}
and $R_{0}$ choosing as $R \ge 2r_{0}$ and $\displaystyle{\frac{R^{2}}{4}} \ge T_{0}$, 
we obtain \eqref{eq:7.10} for $R \ge R_{0}$ and the proof of Lemma \ref{Lem:7.1} is now complete.
\end{proof}
The second step is our new ingredient in this proof.
Roughly speaking, the lemma implies that we can estimate the contribution of the viscoelastic term.    
\begin{lem} \label{Lem:7.2}
Let $u_{0} \in W^{2,1} \cap W^{2, \infty}$.
Then it holds that
\begin{equation} \label{eq:7.11}
\lim_{R \to \infty}
\int_{B_{R}(0)} (-\Delta u_{0}(x) ) \phi_{R}^{p'} dx =0.
\end{equation}
\end{lem}
\begin{proof}
We recall that 
\begin{equation*}
\int_{\R^{n}} \Delta u_{0}(x)  dx =0
\end{equation*}
to see that for an arbitrary $\varepsilon >0$,
there exists $R_{1}>0$ such that
\begin{equation*} 
\left| \int_{B_{\frac{R}{2}}(0)} \Delta u_{0}(x)  dx\right| < \frac{\varepsilon}{2}
\end{equation*}
for all $R \ge R_{1}$.
Moreover, since $\Delta u_{0} \in L^{1}$, 
there exists $R_{2}>0$ such that
\begin{equation*} 
\int_{\R^{n} \setminus (B_{\frac{R}{2}}(0))} |\Delta u_{0}(x)|  dx < \frac{\varepsilon}{2}
\end{equation*}
for all $R \ge R_{2}$.
Therefore, by observing the support of $\phi_{R}$,
we choose $R$ satisfying $R \ge \max\{ R_{1}, R_{2}\}$ to obtain 
\begin{equation*} 
\begin{split}
& \left| 
\int_{B_{R}(0)} (-\Delta u_{0}(x) ) \phi_{R}^{p'} dx
\right| \\
& \le \left| 
\int_{B_{\frac{R}{2}}(0)} (-\Delta u_{0}(x) ) \phi_{R}^{p'} dx
\right| 
+
\left|
\int_{B_{R}(0) \setminus (B_{\frac{R}{2}}(0))} |\Delta u_{0}(x)|  dx 
\right|
 \\
& \le \left| 
\int_{B_{\frac{R}{2}}(0)} \Delta u_{0}(x)  dx
\right| 
+
\left|
\int_{\R^{n} \setminus (B_{\frac{R}{2}}(0))} |\Delta u_{0}(x)|  dx 
\right|
< \frac{\varepsilon}{2}+ \frac{\varepsilon}{2}= \varepsilon,
\end{split}
\end{equation*}
which means the desired estimate \eqref{eq:7.11} and the lemma follows.
\end{proof}
The third lemma states the upper bound of $K_{R}$ on $R$.
\begin{lem} \label{Lem:7.3}
Let $u$ be a solution of \eqref{eq:1.1} and let $p>1$ satisfying $\displaystyle{\frac{1}{p}} +\displaystyle{\frac{1}{p'}} =1$.
Then there exists a constant $C>0$ such that
\begin{align} \label{eq:7.12}
& K_{R}\le C K_{R}^{\frac{1}{p}} R^{\frac{n+2}{p'}-2}
+\left|\int_{B_{R}(0) } (-\Delta u_{0}(x)) \phi_{R}^{p'} dx \right|, \\
& K_{R}\le C K_{R}^{\frac{1}{p}} R^{\frac{n+2}{p'}-4} 
+ CR^{\frac{n+2}{p'} -2}
\left(
\int_{0}^{R^{2}}
\int_{B_{R}(0) \setminus B_{\frac{R}{2}}(0)} 
|u|^{p} dxdt
\right)^{\frac{1}{p}} \label{eq:7.13}
\end{align}
for $R \ge 1$.
\end{lem}
\begin{proof}
We first recall the definition of $K_{R}$ to see 
\begin{equation} \label{eq:7.14}
\begin{split}
K_{R}[u] & =
\int_{0}^{R^{2}} \int_{B_{R}(0) } 
\partial_{t}^{2} u \phi_{R}^{p'} \eta_{R}^{p'} dx dt
+
\int_{0}^{R^{2}} \int_{B_{R}(0) } (-\Delta u) \phi_{R}^{p'} \eta_{R}^{p'} dx dt \\
& +
\int_{0}^{R^{2}} \int_{B_{R}(0) } 
\partial_{t} u \phi_{R}^{p'} \eta_{R}^{p'} dx dt
+
\int_{0}^{R^{2}} \int_{B_{R}(0) } (-\Delta \partial_{t}u) \phi_{R}^{p'} \eta_{R}^{p'} dx dt\\
& =: L_{1}+L_{2}+L_{3}+L_{4}
\end{split}
\end{equation}
by the equation \eqref{eq:1.1}.
In what follows, we estimate $L_{j}$ for $j=1,2,3,4$, respectively.
Note that Zhang \cite{Z} has already obtained 
\begin{align}
& L_{1} \le CR^{\frac{n+2}{p'} -4}I_{R}^{\frac{1}{p}},  \label{eq:7.15} \\
& L_{2} +L_{3}  \le CR^{\frac{n+2}{p'} -2}I_{R}^{\frac{1}{p}},  \label{eq:7.16} \\
& L_{2} +L_{3}  \le CR^{\frac{n+2}{p'} -2} 
\left(
\int_{0}^{R^{2}}
\int_{B_{R}(0) \setminus B_{\frac{R}{2}}(0)} 
|u|^{p} dxdt
\right)^{\frac{1}{p}}  \label{eq:7.17}
\end{align}
by using the estimates \eqref{eq:7.2} - \eqref{eq:7.7}.
So, it suffices to obtain the estimate for $L_{4}$ to finalize the proof of Lemma 7.3. Indeed, it follow from the integration by parts, \eqref{eq:7.11} and the H\"older inequality that 
\begin{equation}  \label{eq:7.18}
\begin{split}
L_{4} 
& = 
\int_{B_{R}(0) } (-\Delta u_{0}(x)) \phi_{R}^{p'} dx 
+
\int_{0}^{R^{2}} \int_{B_{R}(0) } u (-\Delta  \phi_{R}^{p'}) \partial_{t} \eta_{R}^{p'} dx dt \\
& \le 
\left|\int_{B_{R}(0) } (-\Delta u_{0}(x)) \phi_{R}^{p'} dx \right|
+\frac{C}{R^{4}}
\int_{\frac{R^{2}}{4}}^{R^{2}} \int_{B_{R}(0) \setminus B_{\frac{R}{2}}(0)} 
|u|\phi_{R}^{p'-1} \eta_{R}^{p'-1} dx dt \\
& \le 
\left|\int_{B_{R}(0) } (-\Delta u_{0}(x)) \phi_{R}^{p'} dx \right|
+\frac{C}{R^{4}} K_{R}^{\frac{1}{p}}
\left( \int_{0}^{R^{2}}  \int_{B_{R}(0) } dx dt\right)^{\frac{1}{p'}} \\
& \le 
\left|\int_{B_{R}(0) } (-\Delta u_{0}(x)) \phi_{R}^{p'} dx \right|
+CK_{R}^{\frac{1}{p}}R^{\frac{n+2}{p'}-4}.
\end{split}
\end{equation}
Therefore, the combination of \eqref{eq:7.14}-\eqref{eq:7.16} and \eqref{eq:7.18} 
implies the desired estimate \eqref{eq:7.12}. 
Similarly, combining the estimates \eqref{eq:7.14}, \eqref{eq:7.15}, \eqref{eq:7.17} and \eqref{eq:7.18}  
yields the estimate \eqref{eq:7.13}, which completes the proof.
\end{proof}
Now we are in a position to describe the proof of Theorem \ref{Thm:1.3}.
\begin{proof}
Our proof of Theorem \ref{Thm:1.3} can be done by contradiction.
Let us assume that the Cauchy problem \eqref{eq:1.1} has a global solution.
Then $K_{R}$ is well-defined for $R > 0$. 
Now we choose $R>1$ sufficiently large such as Lemmas \ref{Lem:7.1} - \ref{Lem:7.3} hold.
When $1< p < 1+\displaystyle{\frac{2}{n}}$, we see 
\begin{equation} \label{eq:7.19}
\frac{n+2}{p'}-2 <0. 
\end{equation}
Then we can apply the estimates \eqref{eq:7.10} and \eqref{eq:7.12} to have 
\begin{equation*}
0<C_{0} \le K_{R}\le C K_{R}^{\frac{1}{p}} R^{\frac{n+2}{p'}-2}
+\left|\int_{B_{R}(0) } (-\Delta u_{0}(x)) \phi_{R}^{p'} dx \right|,
\end{equation*}
namely we obtain 
\begin{equation} \label{eq:7.20}
\begin{split}
0<C_{0}^{\frac{1}{p}} \le K_{R}^{\frac{1}{p'}} 
& \le C  R^{\frac{n+2}{p'}-2}
+K_{R}^{-\frac{1}{p'}}  \left|\int_{B_{R}(0) } (-\Delta u_{0}(x)) \phi_{R}^{p'} dx \right| \\
 & \le C  R^{\frac{n+2}{p'}-2}
+C_{0}^{-\frac{1}{p'}}  \left|\int_{B_{R}(0) } (-\Delta u_{0}(x)) \phi_{R}^{p'} dx \right|.
\end{split}
\end{equation}
Using \eqref{eq:7.11} and \eqref{eq:7.19}, 
the right hand sides of \eqref{eq:7.20} tends to $0$ as $R \to \infty$, 
which contradicts the lower bound of $K_{R}$.

For the case $p=1+\displaystyle{\frac{2}{n}}$, i.e. $\displaystyle{\frac{n+2}{p'}}-2 =0$,
we first use this to see 
\begin{equation*} 
\int_{0} \int_{\R^{n}} |u|^{p} dxdt < C_{1}, 
\end{equation*}
which means 
\begin{equation} \label{eq:7.21}
\lim_{R \to \infty} 
\int_{0}^{R^{2}}
\int_{B_{R}(0) \setminus B_{\frac{R}{2}}(0)} 
|u|^{p} dxdt	=0.
\end{equation}
By the estimates \eqref{eq:7.10}, \eqref{eq:7.13} and \eqref{eq:7.21},
we arrive at the estimate
$$
0<C_{0} \le 
K_{R}\le C_{1}^{\frac{1}{p}} R^{-2} 
+
\left(
\int_{0}^{R^{2}}
\int_{B_{R}(0) \setminus B_{\frac{R}{2}}(0)} 
|u|^{p} dxdt
\right)^{\frac{1}{p}} \to 0
$$
as $R \to 0$, 
which contradicts the lower bound of $K_{R}$ again. This proves Theorem \ref{Thm:1.3}.
\end{proof}
\vspace*{5mm}
\noindent
\textbf{Acknowledgments. }
%\smallskip
The work of the first author (R. IKEHATA) was supported in part by Grant-in-Aid for Scientific Research (C)15K04958 of JSPS.
The work of the second author (H. TAKEDA) was supported in part by Grant-in-Aid for Young Scientists (B)15K17581 of JSPS.

%%%%%%%%%%%%%%%%%%%%%%%%%%%%%%%%%%%%%%%%%%%%%%%%%%%%%%%%%%%%%%%%%%%%%%%%%%%%%%%%%%%%%%%%%%%%%%%%%%%%%%%%%%%%%%%%%%%%%%%%%%%%%%%%%%%%%%%%%%%%%%%

%%%%%%
%%%%%%%%%
%%%%%%%%%%%
\end{document}